\documentclass[10pt,reqno]{amsart}
\usepackage{a4wide}

\usepackage{amsmath,amsfonts,amsthm,amssymb,amsxtra,dsfont}
\usepackage{amsxtra, amssymb, mathrsfs, pstricks}
\usepackage[english]{babel}
\usepackage{amsmath,amsfonts,amstext,amsbsy,amsopn,amssymb,amsthm,amscd}
\usepackage{amsxtra,latexsym}
\usepackage{exscale}
\usepackage{graphicx,mathrsfs}
\usepackage{psfrag}
\usepackage{paralist,eucal,enumerate}
\usepackage{color,graphics}

\usepackage[colorlinks=true]{hyperref}
\hypersetup{urlcolor=blue, citecolor=blue, linkcolor=blue}

\usepackage{subcaption}

\newtheorem{theorem}{Theorem}[section]
\newtheorem{proposition}[theorem]{Proposition}

\newtheorem{lemma}[theorem]{Lemma}

\theoremstyle{definition}

\newtheorem{definition}[theorem]{Definition}

\theoremstyle{remark}
\newtheorem{remark}[theorem]{\bf Remark}

\numberwithin{equation}{section}

\usepackage{graphicx,mathrsfs}
\newcommand{\field}[1]{\mathbb{#1}}
\newcommand{\N}{\field{N}}
\newcommand{\R}{\field{R}}

\def\bbV{{\mathbb V}}  
 
 \def\calE{{\mathcal E}} \def\calF{{\mathcal F}}
  \def\calI{{\mathcal I}}
  \def\calL{{\mathcal L}}
  
  \def\calR{{\mathcal R}}
  
\def\calV{{\mathcal V}} \def\calW{{\mathcal W}} \def\calX{{\mathcal X}}
 \def\calZ{{\mathcal Z}}
   
\def\rmd{{\mathrm d}} \def\rme{{\mathrm e}}

\def\rmm{{\mathrm m}}   
\def\rmp{{\mathrm p}}

\def\rmD{{\mathrm D}}

\newcommand{\wt}{\widetilde}

 \DeclareMathOperator{\dist}{dist}

\DeclareMathOperator{\Argmin}{Argmin}

\DeclareMathOperator{\diam}{diam}
\DeclareMathOperator{\Lin}{Lin}
\DeclareMathOperator{\Var}{Var}
 \renewcommand{\d}{{\mathrm d}}
 
 \newcommand{\ds}{\,\d s}
 \newcommand{\dt}{\,\d t}

 \newcommand{\dr}{\,\d r}

\newcommand{\abs}[1]{\left\lvert#1\right\rvert}      

 \newcommand{\norm}[1]{\left\lVert#1\right\rVert}

\newcommand{\bnorm}[1]{\bigl\Vert#1\bigr\Vert} 
\newcommand{\Bnorm}[1]{\Bigl\Vert#1\Bigr\Vert} 

\newcommand{\Set}[2]{\{\,#1\,;\,#2\,\}}

\newcommand{\vvp}{\rmp}
\newcommand{\vve}{\rme} 
\newcommand{\vvm}{\rmm} %

\begin{document}

\title[ Parameterized BV-solutions  with 
discontinuous loads]{Existence of parameterized BV-solutions\\ for 
rate-independent systems\\ 
with discontinuous loads}

\author{Dorothee Knees}
\address{Dorothee Knees, Institute of Mathematics, University of Kassel,
Heinrich-Plett Str.~40, 34132 Kassel, Germany. Phone:
+49 0561 8044355}
\email{dknees@mathematik.uni-kassel.de}

\author{Chiara Zanini}
\address{Department of Mathematical Sciences, Politecnico di Torino, 
Corso Duca degli Abruzzi 24, 
10129 Torino, Italy}
\email{chiara.zanini@polito.it}

\date{}

\begin{abstract}
We study a rate-independent system  with non-convex energy and in the 
case of a time-discontinuous loading. We prove existence of 
the rate-dependent viscous regularization by time-incremental problems,  while 
the existence of the so called parameterized $BV$-solutions is obtained via 
vanishing viscosity in a suitable parameterized setting. In addition, we prove 
that the solution set is compact. 
\end{abstract}

\keywords{rate-independent system; discontinuous load; 
parameterized BV-solution; time-incremental minimum problems; vanishing 
viscosity limit.}

\subjclass[2010]{%
Primary: 35R05, 
49J40 
; Secondary: 74C05, 
35Q74, 
35D40, 
49J45. 
}

\maketitle


{\small \tableofcontents}
\section{Introduction}
\label{s:introduction}

In this paper the existence of a solution $z:[0,T]\to\calZ$ of a doubly 
nonlinear problem of the type
\begin{align}
\label{intro.01}
 0\in \partial\calR(\partial_t z(t)) + \rmD\calI(z(t)) - \ell(t),\quad 
z(0)=z_0,\quad t\in [0,T]
\end{align}
is addressed. The focus is on rate-independent systems and hence we assume that 
the dissipation functional  $\calR$ is convex and positively homogeneous of 
degree one. It is further assumed  that the energy functional $\calI$ is 
nonconvex and that the load term $\ell$ is discontinuous in time. It is well 
known that even if $\ell$ is smooth in time, due to the non-convexity of 
$\calI$ the system in general has solutions that are discontinuous in time and 
that also in general there is no uniqueness (see \cite{MieRou15} and references 
therein). In our setting here, a second source for discontinuities is 
introduced 
by the discontinuous load term. We prove the existence of (parameterized) 
balanced viscosity solutions via a vanishing viscosity analysis (Theorem 
\ref{ex.ri.pparam}) and study the compactness of the solution set (Proposition  
\ref{prop.properties_sol_set}). The analysis is carried out in the semilinear 
rate-independent setting introduced in 
\cite{MielkeZelik_ASNPCS14}, compare also \cite[Example 
3.8.4]{MieRou15}, \cite{Knees2018,KneesThomas2018}. 

For a more detailed presentation of the arguments let $\calZ$, $\calV$ be 
Hilbert spaces and $\calX$ a Banach space such that 
$\calZ\Subset\calV\subset\calX$ (compact and continuous embeddings, 
respectively). The dissipation functional $\calR:\calX\to[0,\infty)$ is convex, 
continuous and positively homogeneous of degree one and it is assumed  
to be  equivalent to the norm on $\calX$. The latter assumption simplifies the 
analysis since then $\partial\calR(0)$ is a bounded subset of $\calX^*$. 
However, this assumption rules out the modeling of damage and other 
unidirectional processes. We work in the semilinear setting where  
$\calI:\calZ\to\R$ is of the structure $\calI(z)=\frac12\langle A z,z\rangle 
+\calF(z)$ with a linear and continuous operator $A\in \Lin(\calZ,\calZ^*)$ 
that  is bounded and symmetric (we refer to Section \ref{suse:infiniteloc} for 
the precise assumptions) and a possibly nonconvex functional 
$\calF:\calZ\to[0,\infty)$ that is of lower order with respect to the quadratic 
term in $\calI$.  The loads $\ell$ are taken from $BV([0,T];\calV^*)$. The 
total energy is given by $\calE(t,z)=\calI(z) - \langle \ell(t),z\rangle$. As 
already mentioned, due to the non-convexity of $\calI$ solutions to 
\eqref{intro.01} are discontinuous in time (even if $\ell$ is continuous). 
Several different notions of weak solutions have been introduced in the recent 
literature (see \cite{MieRou15} and references therein) allowing for 
discontinuous solutions, among them the (global) energetic solutions and 
balanced viscosity solutions (BV-solutions). Let us remark that the solution 
concepts are not equivalent.  
Existence of the different solution concepts was 
obtained for more regular data, while the novelty in this paper is to consider 
the case of BV-loading. Existence is studied via vanishing viscosity  
resulting in  BV-solutions. For that purpose, we  consider the  regularized 
problem
\begin{equation}\label{three}
   0\in \partial\calR(\partial_t z_\varepsilon(t)) + \varepsilon\bbV\partial_t 
z_\varepsilon(t) + \rmD_z \calE(t,z_\varepsilon(t)),\quad 
z_\varepsilon(0)=z_0,\ 
\ t\in[0,T]
\end{equation}
obtained by adding the viscous term $\varepsilon\bbV\partial_t z(t)$ ($\bbV$ is 
a linear operator) to \eqref{intro.01} with the  parameter 
$\varepsilon>0$. After having established the existence and uniqueness of 
solutions to the regularized problem (Proposition \ref{prop.unifboundeps}) we 
study the limit $\varepsilon\to 0$. In order to perform the vanishing viscosity 
analysis, the inclusion \eqref{three} is rewritten in a parameterized version, 
i.e. $t\mapsto z_\varepsilon(t)$ is replaced with $s\mapsto (\hat 
t_\varepsilon(s),\hat z_\varepsilon(s))$, where $\hat 
z_\varepsilon(s))=z_\varepsilon(\hat t_\varepsilon(s))$. There are different 
possibilities for choosing the parameterization. We take here the 
paramterization based on the vanishing viscosity contact potential 
(\cite{MRS16}, see 
\eqref{s_eps}). The advantage of this choice  is that viscosity limits 
automatically are normalized in the parameterized picture (see 
\eqref{eq.normalized.sol}). In the convergence proofs we closely follow the 
arguments in \cite{MRS16} and adapt them to our situation. Due to the 
semilinear 
structure of our problem, some stronger statements in particular concerning the 
regularity of solutions (e.g.\ $\rmD\calE \in \calV^*$ instead of $\calZ^*$) 
compared to those in \cite{MRS16} are possible. 
Due to the possible discontinuities of the load term $\ell$ a refined analysis 
of the power term $\int_0^t\langle\ell(r),\partial_t 
z_\varepsilon(r)\rangle\dr$ and its reparameterized version is necessary. 
Observe that in the reparameterized version the function  
$s\mapsto \ell(\hat t_\varepsilon(s))$ appears. Interpreting the power term as 
a Kurzweil integral the limit $\varepsilon\to 0$ can be 
identified. We refer to  \cite{KrejciLiero09} (and Appendix \ref{sec.Kurzweil}) 
for an overview on the properties of the Kurzweil integral. 

In order to perform the  vanishing viscosity analysis, estimates for 
solutions to \eqref{three} are needed that are uniform with respect to the 
viscosity parameter $\varepsilon$. Due to the low regularity of the load term 
$\ell$,  arguments from the literature are not directly 
applicable  since there it is typically assumed and used that $\ell$ has 
temporal $H^1$ or $C^1$-smoothness. The new estimates are stated in 
Propositions 
\ref{prop.basicest} and 
\ref{prop.BV-estimate}. As a new feature these estimates do not depend on the 
length of the time interval $[0,T]$ and the constants in the estimates are 
scaling invariant. This allows for instance to transfer estimates by rescaling 
arguments to different time intervals without changing the constants. This 
observation is exploited in the analysis of solution sets to the system 
\eqref{intro.01}, see Proposition \ref{prop.properties_sol_set}.

This is not the first paper that investigates solutions to rate-independent 
systems with discontinuous loads. 
Let us first mention the article \cite{KrejciLiero09} that is closest to our 
investigations. In contrast to our setting, in \cite{KrejciLiero09}  the energy 
$\calE(t,\cdot)$ is assumed to be strictly convex in $z$ and the dissipation 
potential $\calR$ may depend in a discontinuous way on the time. Starting from 
a 
time incremental minimization problem (without adding additional viscosity) the 
authors prove the existence and uniqueness of solutions within their solution 
class. In addition, if $\calE$ is quadratic, they compare this solution with 
the 
one obtained from a vanishing viscosity analysis. The analysis is carried out 
in 
the physical time and integrals over time intervals are interpreted in the 
Kurzweil sense. 
A  
different  approach  was followed in \cite{Recupero11,RecuperoJCA16} based on 
measure theory tools, and originally was  
developed for  the study of the mapping properties of the  play 
operator, solving variational inequalities associated to sweeping processes 
\cite{Moreau77,KrejciLaurencot}. More precisely, in 
\cite{Recupero11,RecuperoJCA16} the existence results from  
\cite{Moreau77} are re-obtained for discontinuous BV-loadings by using the 
following steps: reparameterize suitably the 
problem by ``filling in the jumps of the loading $\ell$'' in order to obtain a 
Lipschitz-setting, use the better regularity to get existence of a solution, 
and 
then parameterize back to the BV-setting via measure theory arguments (instead 
of time discretization procedure \cite{Moreau77}). This approach works  thanks 
to the fact that 
sweeping processes are rate-independent. The underlying energies in general 
are convex but the set of admissible forces is allowed to depend on time in a 
discontinuous way, \cite{RecuperoSantambrogio18}. 
Translated to our setting this  means that 
$\calR$ in addition depends on the time and that $t\mapsto \calR(t,z)$ is of 
bounded variation. It is shown in 
\cite{Recupero11} that the solution $z$ depends on the parameterization chosen, 
in the sense that, by using segments (geodesics) to fill in the jumps of 
$\ell$, 
one may get a solution different from the vanishing viscosity 
one. We refer to \cite{KrejciRecupero14} for a comparison of the different 
solution concepts. Clearly, a comparison of the parameterized BV-solutions 
derived in this paper with the above mentioned results would clarify the 
relations between all these different approaches. This would require to 
translate back our solutions to the physical time. Due to the length of this 
paper we postpone this comparison to a future paper.

The paper is organized as follows: in Sec.~\ref{suse:infiniteloc} the precise 
assumptions are settled and the basic  and enhanced estimates are derived in 
order to do the limiting analysis. In Sec.~\ref{s-viscousE} we pass to the 
limit 
in the time incremental viscous problems (expressed as usual in this context 
via energy 
balance) and derive existence and uniqueness of solution for $\varepsilon>0$ 
fixed. Then in Sec.~\ref{s-epsto0}, to perform the vanishing viscosity analysis 
$\varepsilon\to0$ we use the reparameterization 
technique originally introduced in \cite{MiEf06} and refined in \cite{MRS16}, 
that  is we rewrite the problem in a suitable parameterized setting, see 
\eqref{s_eps}, and pass to the limit as $\varepsilon\to0$ in this setting. 
Finally, in Sec.~\ref{s-properties} we discuss the regularity properties and 
compactness of the set of ($\vvp$)-parameterized solutions. The paper closes 
with an appendix where basic facts about the Kurzweil integral, about 
absolutely continuous functions and BV-functions and a chain rule are 
collected. 

\section{Basic assumptions and estimates for a time-incremental scheme} 
\label{suse:infiniteloc}
Let $\calX$ be a Banach 
space and $\calZ,\calV$ be separable Hilbert spaces that are densely and 
compactly, resp.\ continuously, embedded in the following way:
\begin{align}
\label{eq.Mief000}
 \calZ\Subset\calV\subset \calX.
\end{align}
Let further $A\in \Lin(\calZ,\calZ^*)$ and $\bbV\in \Lin(\calV,\calV^*)$ be 
linear symmetric, bounded  $\calZ$- and $\calV$-elliptic operators, i.e.\ there 
exist constants $\alpha,\gamma>0$ such that 
\begin{align}
\label{eq.Mief0001}
 \forall z\in \calZ, \forall v\in \calV:\quad 
 \langle Az,z\rangle\geq \alpha\norm{z}^2_\calZ\,,\quad 
\langle \bbV v,v\rangle\geq \gamma\norm{v}^2_\calV\,,
\end{align}
and $\langle Az_1,z_2\rangle =\langle A z_2,z_1\rangle$ for all $z_1,z_2\in 
\calZ$ (and similar for $\bbV$). 
Here, $\langle \cdot,\cdot\rangle$ denotes the duality pairings in $\calZ$ and 
$\calV$, respectively. We define $\norm{v}_\bbV:=\left(\langle \bbV 
v,v\rangle\right)^\frac{1}{2}$, which is a norm that is equivalent to the 
Hilbert space norm $\norm{\cdot}_\calV$. 
Let further 
\begin{align}
\label{eq.Mief00a}
  \calF\in C^2(\calZ;\R)\text{  with } 
\calF\geq 0.
\end{align} 
The functional $\calF$ shall 
play the role of a possibly nonconvex lower order term (cf.\ \cite[Section 
3.8]{MieRou15}). Hence, we assume that  
\begin{align}
\label{ass.F01}
 \rmD\calF\in C^1(\calZ;\calV^*),\quad \norm{\rmD^2\calF(z)v}_{\calV^*}\leq 
C(1 + \norm{z}_\calZ^q)\norm{v}_\calZ
\end{align}
for some $q\geq 1$. 
For the load we assume 
\begin{align}
 \label{eq.Mief00b}
 \ell\in BV([0,T];\calV^*)\,, 
\end{align}
and 
\[
\Var_{\calV^*}(\ell,[a,b])=\sup_\text{partitions $(t_k)$ of 
$[a,b]$}\sum_k  
\norm{\ell(t_k)-\ell(t_{k-1})}_{\calV^*}
\]
denotes the total variation of $\ell$ on 
$[a,b]$ with respect to $\calV^*$.  

Energy functionals of the following type are considered 
\begin{align}
\label{eq.Mief0002}
 \calI:&\calZ \rightarrow \R,\quad \calI(z):=\frac{1}{2}\langle A 
z,z\rangle + 
\calF(z) ,
\\
\calE:&[0,T]\times\calZ\to\R, \quad \calE(t,z)=\calI(z) - \langle 
\ell(t),z\rangle\,.
\end{align}
Clearly, $\calI\in C^1(\calZ;\R)$. 

The dissipation functional  $\calR:\calX\to[0,\infty)$ is assumed to be  
convex, continuous, positively homogeneous 
of degree one and 
\begin{align}
\label{eq.Mief100}
 \exists c,C>0\, \forall x\in \calX:\quad c\norm{x}_\calX\leq \calR(x)\leq 
C\norm{x}_\calX\,.
\end{align}
We refer to Appendix \ref{app.R} for the properties of $\calR$ which 
will be used in the following. 
From \eqref{ass.F01} and \eqref{eq.Mief100} we deduce the following 
interpolation estimate, \cite[Lemma 1.1]{Knees2018}: 
\begin{lemma}
 \label{lem.estDF}
 Assume \eqref{eq.Mief000}, \eqref{eq.Mief00a}, \eqref{eq.Mief100} and 
\eqref{ass.F01}. For every 
$\rho>0$ and $\kappa>0$ there exists $C_{\rho,\kappa}>0$ such that 
for all $z_1,z_2\in \calZ$ with $\norm{z_i}_\calZ\leq \rho$ we have 
\begin{multline}
 \label{est:DF}
 \abs{\langle \rmD\calF(z_1)-\rmD\calF(z_2),z_1 - z_2\rangle}
 \\
 \leq \kappa 
\norm{z_1-z_2}^2_\calZ + 
C_{\rho,\kappa}\min\{\calR(z_1-z_2),\calR(z_2 - z_1)\}\norm{z_1-z_2}_\bbV.
\end{multline} 
\end{lemma}
As a consequence, $\calE$ is  $\lambda$-convex on sublevels. To be more 
precise, we have the following estimate: For every $\rho>0$ there exists 
$\lambda=\lambda(\rho)>0$ such that for all $t\in [0,T]$ and all $z_1,z_2\in 
\calZ$ with $\norm{z_i}_\calZ\leq \rho$ we have
\begin{align}
 \label{est.lambda-convex}
 \langle \rmD_z\calE(t,z_1) - \rmD_z\calE(t,z_2), z_1 - 
z_2\rangle_{\calZ^*,\calZ}\geq \tfrac{\alpha}{2}\norm{z_1 - z_2}_\calZ^2 - 
\lambda\norm{z_1 - z_2}_\bbV^2 
\end{align}
and
\begin{align}
 \label{est.lambda-convex-energy}
 \calI(z_2) - \calI(z_1) \geq \langle \rmD\calI(z_1), z_2 - 
z_1\rangle_{\calZ^*,\calZ} + \tfrac{\alpha}{2}\norm{z_1 - z_2}^2_\calZ
- \lambda \calR(z_2 -z_1)\norm{z_2 - z_1}_\calV.
\end{align}
In the following we replace $\rmD_z\calE(t,z)$ by $\rmD\calE(t,z)$
so that
\[
   \rmD\calE(t,z) = \rmD\calI(z) - \ell(t) = Az + \rmD\calF(z) - \ell(t).
\]
For the proof of the existence theorems we need a further 
assumption on $\calF$: 
\begin{align}
\label{ass.fweakconv}
 \calF:\calZ\rightarrow\R \text{ and }\rmD\calF:\calZ\rightarrow\calZ^* 
\text{ are weak-weak continuous.}
\end{align}
%

In the next lemma we prove a coercivity estimate  for $\calE$ and a 
product estimate which will be used to derive a uniform estimate on 
$\norm{z_k^N}_\calZ$, see Proposition~\ref{prop.basicest} below. Similar 
arguments were used in the proof of \cite[Lemma 3.1]{KrejciLiero09}. 
\begin{lemma}
 \label{lem:basic-energy-estimates}
 Assume \eqref{eq.Mief000}--\eqref{eq.Mief00b}. 
 
 Let $c_0:=\frac{c_\calZ^2}{\alpha}(1 +\norm{\ell}^2_{L^\infty(0,T;\calV^*)})$, 
where $c_\calZ$ is the embedding constant for $\calZ\subset\calV$. Then for 
every $t\in [0,T]$  and $v\in \calZ$ we have
\begin{align}
 \calE(t,v) + c_0&\geq c_\calZ\norm{v}_\calZ\geq  
\norm{v}_\calV\,.\label{eq.basic1} 
\end{align}
A product estimate: Let $\Set{a_k}{1\leq k\leq N}$  with $a_k\geq0$ for 
every $k$, and $c>0$. Then
\[
 \prod_{k=1}^N(1 + ca_k)\leq \exp\left(c\sum_{k=1}^N a_k\right)\,.
\]
As a consequence, let $c>0$, $\ell\in BV([0,T];\calV^*)$ and let $0\leq t_0<t_1 
<\ldots<t_N\leq T$ be an arbitrary partition of $[0,T]$. Then 
\begin{align}
\label{est.prod1}
\prod_{k=1}^N(1 + 
c\norm{\ell(t_k)-\ell(t_{k-1})}_{\calV^*})\leq\exp\left(c\Var_{\calV^*}(\ell,[
t_0 , t_N ] )\right)\,.
\end{align}
\end{lemma}

\begin{proof}
 Let $t\in [0,T]$, $v\in \calZ$. By coercivity and Young's inequality
 \begin{align*}
  \calE(t,v)&\geq \tfrac{\alpha}{2}\norm{v}_\calZ^2  
-c_\calZ\norm{\ell(t)}_{\calV^*}\norm{v}_\calZ 
\geq \tfrac{\alpha}{4}\norm{v}^2_\calZ - 
\tfrac{c_\calZ^2}{\alpha}\norm{\ell}_{L^\infty(0,T;\calV^*)}^2\,.
 \end{align*}
Together with 
$\norm{v}_\calV \leq c_\calZ\norm{v}_\calZ \leq \tfrac{c_\calZ^2}{\alpha} + 
\tfrac{\alpha}{4}\norm{v}_\calZ^2$ one obtains \eqref{eq.basic1}.

Proof of the product estimate: Since for $y\geq 0$ we have $\ln (1+y)\leq y$, 
it holds
\begin{align*}
 \prod_{k=1}^N(1 + ca_k)=
\exp(\sum_{k=1}^N\ln(1 + ca_k))\leq \exp(c\sum_{k=1}^Na_k)\,. 
\end{align*}
\end{proof}


We consider viscous regularizations of the 
rate-independent system $(\calE,\calR,\calZ)$ with respect to the 
intermediate space $\calV$. For $\varepsilon\geq 
0$ let
\begin{align*}
 \calR_\varepsilon:\calV\to[0,\infty), \quad \calR_\varepsilon(v):=\calR(v) 
+\tfrac{\varepsilon}{2}\langle \bbV v,v\rangle\,. 
\end{align*}
Properties about $\calR_\varepsilon$, $\varepsilon\geq0$, are collected 
in the Appendix~\ref{app.R}.

We start from the usual time-incremental minimization problems: Let 
$0=t_0<t_1<\ldots <t_N=T$ be an arbitrary partition of $[0,T]$ and let 
$\tau_k:= t_{k} - t_{k-1}$, for $k=1,\ldots,N$. With $z^N_0:=z_0$, for 
$k=1,\ldots,N$ define $z_k^N$ 
recursively via
\begin{align}
\label{eq.time-incr}
 z_k^N\in \Argmin\Set{\calE(t_k,v) + \tau_k\calR_\varepsilon\left( 
(v- z_{k-1}^N)/\tau_k\right)}{v\in \calZ}\,.
\end{align}
Minimizers exist by the direct method in the calculus of variations. 
In the next proposition we collect the basic estimates for the time-incremental 
minimization problems. 
\begin{proposition}
 \label{prop.basicest}
Under the above conditions on $\calE$ and $\calR_\varepsilon$ there exists a 
constant $C>0$ such that 
for all  $\varepsilon\geq 0$, $N\in \N$ and $1\leq k\leq N$ we have, with $c_0$ 
from Lemma~\ref{lem:basic-energy-estimates},
\begin{align}
 \norm{z_k^N}_\calZ &\leq 
 c_\calZ^{-1} (\calE(0,z_0) + c_0) 
\exp(\Var_{\calV^*}(\ell,[0,t_k]))\,,\label{est.unifz1}
\\
0\leq c_0 + \calE(t_k,z_k^N) &\leq (\calE(0,z_0) + c_0) 
\exp(\Var_{\calV^*}(\ell,[0,t_k]))\,, 
\label{est.E}
\\
\sum_{s=1}^N\tau_s\calR_\varepsilon((z_s^N - z_{s-1}^N)/\tau_s) &\leq \wt C\,
\label{est.unifdissip1}
\end{align}
with $\wt C=(\calE(0,z_0) + c_0)\Big(1 
+ \Var_{\calV^*}(\ell,[0,T])\exp(\Var_{\calV^*}(\ell,[0,T]))\Big)$.
The following energy-dissipation estimates are valid 
\begin{align}
 \calE(t_k,z_k^N) 
 + \sum_{s=1}^k\tau_s\calR_\varepsilon((z_s^N - 
z_{s-1}^N)/\tau_s) 
&\leq \calE(t_0,z_0) +\sum_{s=1}^k\langle \ell(t_{s-1}) - \ell(t_s), 
z_s^N\rangle_{\calV^*,\calV}\,,
\label{est.engdiss1}\\
\calI(z_k^N) + \sum_{s=1}^k\tau_s\calR_\varepsilon((z_s^N - z_{s-1}^N)/\tau_s) 
&\leq \calI(z_0) + \sum_{s=0}^{k-1}\langle \ell(t_s) , z_{s+1}^N - z_s^N 
\rangle_{\calV^*,\calV}\,.
\label{est.engdiss2}
\end{align}
\end{proposition}

\begin{proof}
  By minimality, we  obtain from \eqref{eq.time-incr} (suppressing the 
  index $N$) 
together with 
\eqref{eq.basic1}
\begin{align}
 \calE(t_k,z_k) & + \tau_k\calR_\varepsilon((z_k - z_{k-1})/\tau_k)\leq 
 \calE(t_{k-1},z_{k-1}) +\langle \ell(t_{k-1}) - \ell(t_k),z_{k-1}\rangle
 \label{est.ed1step}\\
 &
 \leq \calE(t_{k-1},z_{k-1}) 
 + \norm{\ell(t_{k-1}) -\ell(t_k)}_{\calV^*} \norm{z_{k-1}}_\calV
 \nonumber \\
&
\leq \calE(t_{k-1},z_{k-1}) 
+ \norm{\ell(t_{k-1})-\ell(t_k)}_{\calV^*} 
\big(c_0 + \calE(t_{k-1},z_{k-1})\big)\,.
\nonumber 
\end{align}
Adding $c_0$ on both sides yields 
\begin{align*}
 \calE(t_k,z_k) + c_0 \leq \big(\calE(t_{k-1},z_{k-1}) + c_0\big) (1+ 
\norm{\ell(t_{k-1})-\ell(t_k)}_{\calV^*} )\,,
\end{align*}
and by recursion and \eqref{est.prod1}
\begin{align*}
 \calE(t_k,z_k) + c_0 &\leq (\calE(t_0,z_0) + c_0 
)\prod_{s=1}^k(1+\norm{\ell(t_s) - \ell(t_{s-1})}_{\calV^*})
\\
&\leq (\calE(t_0,z_0) + c_0) \exp(\Var_{\calV^*}(\ell,[0,t_k]))\,.
\end{align*}
Together with \eqref{eq.basic1} we arrive at \eqref{est.unifz1} and 
\eqref{est.E}. 
The energy dissipation estimate \eqref{est.engdiss1} follows from 
\eqref{est.ed1step},   again by recursion, while  estimate \eqref{est.engdiss2} 
is nothing else but a consequence of discrete integration by parts in the power 
term. 
Since 
\[ 
 \abs{\sum_{s=1}^k\langle \ell(t_{s-1}) - \ell(t_s), 
z_s\rangle_{\calV^*,\calV}}\leq c_\calZ 
\Var_{\calV^*}(\ell,[0,T])\sup_k\norm{z_k}_\calZ,
\]
from \eqref{est.engdiss1} and \eqref{eq.basic1} (i.e.\ 
$\calE(t_k,z_k)\geq -c_0$)  we finally obtain \eqref{est.unifdissip1}. 
\end{proof}

\begin{remark}
 Let $\triangle_N:=\max\Set{t_k-t_{k-1}}{1\leq k\leq N}$ denote the fineness of 
the partition of $[0,T]$. There exists $m>0$ such that the minimizers $z_k^N$ 
of 
\eqref{eq.time-incr} are 
unique provided that  $\varepsilon>m \triangle_N$.  
Indeed, by \eqref{est.unifz1} the minimizers $z_k^N$ are uniformly bounded with 
respect  to $\varepsilon\geq 0$ and the partitions of $[0,T]$, and they satisfy 
the inclusion 
$0\in \partial\calR(z_k^N -z_{k-1}^N) + \tfrac{\varepsilon}{\tau_k}\bbV(z_k^N - 
z_{k-1}^N )+ \rmD\calE(t_k^N,z_k^N)$. The maximal monotonicity of 
$\partial\calR$ in combination with estimate \eqref{est.lambda-convex} implies 
uniqueness provided that $\varepsilon/\triangle_N>\lambda$ with $\lambda$ from 
\eqref{est.lambda-convex}. 
\end{remark}

In order to carry out the vanishing viscosity analysis we need more refined 
estimates. In the following $\dist_\bbV(\cdot,\partial\calR(0))$  
denotes the distance of an element of $\calV^*$ to 
$\partial\calR(0)\subset\calV^*$, see~\eqref{distV}. 

\begin{proposition}
 \label{prop.BV-estimate} 
 Assume \eqref{eq.Mief000}--\eqref{eq.Mief100}. 
 Assume in addition that $\rmD\calE(0,z_0)\in \calV^*$.  Then
 for all $\varepsilon\geq 0$, all $N\in 
\N$ and all 
partitions $\Pi_N$ of $[0,T]$ we have 
\begin{align}
\sum_{k=1}^N\norm{z_k^N - z_{k-1}^N}_\calZ
+\sup_{1\leq k\leq N} 
\tfrac{\varepsilon}{\tau_k^N}\norm{z_k^N- z_{k-1}^N}_\bbV
&\leq 
C_1
\label{est.bv01}\\
\sup_{1\leq k\leq N}\norm{\rmD\calE(t_k^N,z_k^N)}_{\calV^*}&\leq 
\diam_{\calV^*}(\partial\calR(0)) +C_1\,,
\label{est.bv02}
\end{align}
where $C_1=\dist_\bbV(-\rmD\calE(0,z_0),\partial\calR(0)) + 
c_\bbV\Var_{\calV^*}(\ell,[0,T]) + C_I\wt C$ with $\wt C$ from 
\eqref{est.unifdissip1} and $C_I=C_{\rho,\kappa}$ from \eqref{est:DF} for 
$\kappa=\alpha/2$ and $\rho$ is the right hand side of   \eqref{est.unifz1}. 
Finally, for every $\varepsilon>0$ there  exists a constant $C_\varepsilon>0$ 
such that for all partitions $\Pi_N$ we have
\begin{align}
\label{est.bv04}
 \sum_{k=1}^N\tau_k\Bnorm{\frac{z_k^N - z_{k-1}^N}{\tau_k}}_\calZ^2 \leq 
C_\varepsilon.
\end{align}
\end{proposition}


\begin{remark}
\label{rem.bvestindepT}
Observe first that the constants $C_1$, $C_I$ and $\wt C$ are independent of 
the partition $\Pi_N$ and of $\varepsilon>0$.    Observe further  that the 
constants appearing in \eqref{est.bv01}--\eqref{est.bv02} 
are invariant with respect to a rescaling in time. 
\end{remark}

\begin{proof}[Proof of Proposition \ref{prop.BV-estimate}] 
Choose a partition $\Pi_N$ of $[0,T]$ and $\varepsilon\geq 0$. 
 Let $\Set{z_k}{1\leq k\leq N}$ be minimizers according to \eqref{eq.time-incr} 
(we omit the index $N$). 
Then for all $1\leq k\leq N$  we 
have 
\begin{align}
\label{eq.discrincl01}
-\xi_k:= -\tfrac{\varepsilon}{\tau_k}\bbV(z_k - z_{k-1}) 
-\rmD \calE(t_k,z_k)\in 
\partial\calR(z_k - z_{k-1})\,.
\end{align}
Due to the convexity and one-homogeneity of $\calR$ we obtain 
$-\calR(z_k - z_{k-1})=\langle \xi_k,z_k-z_{k-1}\rangle $ and $\calR(z_k 
-z_{k-1}) \geq \langle -\xi_{k-1}, z_k-z_{k-1}\rangle$, see 
Appendix~\ref{app.R}. 
Hence, after adding these relations and rearranging the terms, 
for $2\leq k\leq N$  
we arrive at
\begin{multline}
\label{eq.discrincl02}
 \tfrac{\varepsilon}{\tau_k}\norm{z_k - z_{k-1}}_\bbV^2 - 
\tfrac{\varepsilon}{\tau_{k-1}}\langle \bbV (z_{k-1}- z_{k-2}), z_k - 
z_{k-1}\rangle + \langle A(z_k - z_{k-1}),z_k - z_{k-1}\rangle 
\\
\leq 
\langle \rmD\calF(z_{k-1}) - \rmD\calF(z_k), z_k - z_{k-1}\rangle +\langle 
 \ell(t_k) - \ell(t_{k-1}), z_k - z_{k-1}\rangle.
\end{multline}
The left hand side can be estimated as 
\begin{align*}
 \text{l.h.s.}\geq \left(\tfrac{\varepsilon}{\tau_k}\norm{z_k - z_{k-1}}_\bbV 
 - \tfrac{\varepsilon}{\tau_{k-1}}\norm{z_{k-1} - z_{k-2}}_\bbV 
 \right)\norm{z_k - z_{k-1}}_\bbV
+\alpha\norm{z_k - z_{k-1}}_\calZ^2,
\end{align*}
where $\alpha>0$ is the constant from \eqref{eq.Mief0001}. 
For the right hand side we 
deduce from Lemma \ref{lem.estDF} (where we choose 
$\kappa=\frac{\alpha}{2}$ and $\rho$ according to the right hand side in 
\eqref{est.unifz1}) 
that 
\begin{align*}
 \text{r.h.s.}\leq \tfrac{\alpha}{2}\norm{z_k - z_{k-1}}_\calZ^2 + C(\calR(z_k 
- z_{k-1}) 
+ \bnorm{\ell(t_k) -  \ell(t_{k-1})}_{\calV^*})
\norm{z_k - z_{k-1}}_\bbV. 
\end{align*}
Observe that  $C>0$ 
is independent of $\varepsilon$ and of the 
partition of $[0,T]$. 
 Joining both inequalities  we obtain for all $k\in \{2,\ldots,N\}$  
 \begin{multline*}
  \tfrac{\varepsilon}{\tau_k}\norm{z_k - z_{k-1}}_\bbV
+\tfrac{\alpha}{2 c_\calZ}\norm{z_k - z_{k-1}}_\calZ
\\
\leq   \tfrac{\varepsilon}{\tau_{k-1}}\norm{z_{k-1}- z_{k-2}}_\bbV 
+ C(\calR(z_k 
- z_{k-1}) + \norm{\ell(t_k) - \ell(t_{k-1})}_{\calV^*}),
 \end{multline*}
 where $c_\calZ>0$ is the embedding constant for $\calZ\subset \calV$. 
Summation with respect to $k$ finally yields (for $2\leq K\leq N$)
\begin{multline}
\label{eq.doro1}
\tfrac{\varepsilon}{\tau_K}\norm{z_K - z_{K-1}}_\bbV +  
\tfrac{c\alpha}{2}\sum_{k=2}^K\norm{z_k - z_{k-1}}_\calZ
\\
\leq 
\tfrac {\varepsilon}{\tau_1}\norm{z_1 - z_0}_\bbV +
C\sum_{k=2}^K (\calR(z_k - z_{k-1}) 
+ \norm{ \ell(t_k) - \ell(t_{k-1})}_{\calV^*})
\end{multline}
 Let now $k=1$. Choose $\mu\in 
\partial\calR(0)$ such that 
\[
\dist_\bbV(-\rmD\calE(0,z_0),\partial\calR(0))=\norm{\mu + 
\rmD\calE(0,z_0)}_\bbV.
\]
Together with \eqref{eq.discrincl01} (for $k=1$) and 
from the one-homogeneity of $\calR$  we obtain 
\begin{align*}
 0&\geq \langle\rmD\calE(t_1,z_1) + \mu, z_1 - z_0\rangle + 
\tfrac{\varepsilon}{\tau_1}\langle \bbV(z_1 - z_0), (z_1 - z_0)\rangle\\
&=\langle \rmD\calE(0,z_0) + \mu, z_1 - z_0\rangle + 
\langle 
\rmD\calE(t_1,z_1) - \rmD\calE(0,z_0), z_1 - z_0\rangle + 
\tfrac{\varepsilon}{\tau_1}\norm{z_1 - z_0}_\bbV^2\,.
\end{align*}
By the structure of $\rmD\calE$ and after rearranging the terms we obtain 
\begin{align}
\label{eq.doro3}
 &\tfrac{\varepsilon}{\tau_1}\norm{z_1 - z_0}_\bbV^2
 + \alpha\norm{z_1 - z_0}_\calZ^2
 \\
 \nonumber 
& \leq 
-\langle\rmD\calE(0,z_0)+\mu,z_1 - z_0\rangle 
+\langle (\rmD\calF(z_0) -\rmD\calF(z_1)) + (\ell(t_1) - \ell(t_0)),z_1 - 
z_0\rangle
\\
\nonumber 
&\leq  \tfrac{\alpha}{2}\norm{z_1 - z_0}^2_\calZ + 
\Big( \dist_\bbV(-\rmD\calE(0,z_0),\partial\calR(0)) 
+ \norm{\ell(t_1) - \ell(t_0)}_{\bbV^*} \big.
\\
&\qquad\qquad\qquad\qquad\qquad\qquad\qquad\qquad\qquad\qquad\big. + C\calR(z_1 
- z_0)\Big)\norm{z_1 - z_0}_\bbV .
\nonumber 
\end{align}
For the last estimate we used the definition of $\mu$ and similar estimates as 
for the case $k\geq 2$. Similar to the case $k\geq 2$ we further obtain
\begin{multline*}
 \tfrac{\varepsilon}{\tau_1}\norm{z_1-z_0}_\bbV +\tfrac{\alpha}{2 
c_\calZ}\norm{z_1 - z_0}_\calZ 
\\
\leq 
\dist_\bbV(-\rmD\calE(0,z_0),\partial\calR(0)) +C\big(\norm{\ell(t_1) - 
\ell(t_0)}_{\calV^*} + \calR(z_1 - z_0)\big)\,.
\end{multline*}
Adding the last estimate to \eqref{eq.doro1} finally results in 
\begin{multline}
\label{eq.doro2}
 \tfrac{\varepsilon}{\tau_K}\norm{z_K - z_{K-1}}_\bbV +\tfrac{\alpha}{2 
c_\calZ}\sum_{k=1}^K\norm{z_k - z_{k-1}}_\calZ
\\
\leq 
\dist_\bbV(-\rmD\calE(0,z_0),\partial\calR(0)) 
+ C \Var_{\calV^*}(\ell,[0,t_K]) 
+ C\sum_{k=1}^K\calR(z_k - z_{k-1}) \,, 
\end{multline}
which is valid for $1\leq K\leq N$. 
 Thanks to Proposition 
\ref{prop.basicest} the right hand 
side is uniformly bounded with respect to $\varepsilon\geq 0$ and the 
partitions 
 of $[0,T]$ and we have shown estimate  \eqref{est.bv01}. 
 
In order to prove \eqref{est.bv02} observe that 
$\partial\calR(0)\subset\calZ^*$ can be identified with a  subset of 
$\calV^*$ that is bounded  with respect to the 
$\calV^*$-norm, see \cite[Lemma A.1]{Knees2018}. 
Hence, for $k\geq 1$ from \eqref{eq.discrincl01} we conclude 
$-\rmD\calE(t_k,z_k)\in 
\partial\calR(0) + \frac{\varepsilon}{\tau_k}\bbV(z_k - z_{k-1})\subset\calV^*$ 
and thus 
by \eqref{est.bv01} we ultimately arrive at \eqref{est.bv02}.

For the proof of \eqref{est.bv04} we start again from \eqref{eq.discrincl02}. 
Using $2a(a-b)=a^2 - b^2 + (a-b)^2$, for the first two terms we obtain after 
dividing by $\tau_k$ for $k\geq 2$
\begin{multline*}
 \tfrac{\varepsilon}{2}\norm{\tfrac{z_k - z_{k-1}}{\tau_k}}_\bbV^2 
+\tfrac{\varepsilon}{2}\norm{\Big(\tfrac{z_k - z_{k-1}}{\tau_{k}}\Big) 
- \Big(\tfrac{z_{k-1} - z_{k-2}}{\tau_{k-1}}\Big)}_\bbV^2
+\alpha\tau_k\norm{\tfrac{z_{k} - z_{k-1}}{\tau_{k}}}^2_\calZ
\\
\leq 
 \tfrac{\varepsilon}{2}\norm{\tfrac{z_{k-1} - z_{k-2}}{\tau_{k-1}}}_\bbV^2 
 +\langle \rmD\calF(z_{k-1}) - \rmD\calF(z_k), \tfrac{z_k - 
z_{k-1}}{\tau_k}\rangle 
+\langle  \ell(t_k) -  \ell(t_{k-1}),\tfrac{z_k - 
z_{k-1}}{\tau_k}\rangle.
\end{multline*}
Summation with respect to $2\leq k\leq N$ and adding 
($\tau_1^{-1}\ast$\eqref{eq.doro3}) yields 
\begin{align}
\tfrac{\varepsilon}{2} \norm{\tfrac{z_{N} - z_{N-1}}{\tau_N}}_\bbV^2
&+\tfrac{\varepsilon}{2} \norm{\tfrac{z_{1} - z_{0}}{\tau_1}}_\bbV^2
\nonumber 
\\
\nonumber
&\phantom{\leq} 
+ \tfrac{\varepsilon}{2}\sum_{k=2}^N \norm{\Big(\tfrac{z_k - 
z_{k-1}}{\tau_{k}}\Big) 
- \Big(\tfrac{z_{k-1} - z_{k-2}}{\tau_{k-1}}\Big)}_\bbV^2
+
\alpha\sum_{k=1}^N \tau_k\norm{\tfrac{z_{k} - z_{k-1}}{\tau_k}}^2_\calZ
\\ 
\nonumber 
&\leq 
-\langle\rmD\calE(0,z_0) + \mu,\tfrac{z_1 - z_0}{\tau_1}\rangle 
\\\nonumber 
&\phantom{\leq}+ 
\sum_{k=1}^N\langle \rmD\calF(z_{k-1}) - \rmD\calF(z_k), \tfrac{z_k - 
z_{k-1}}{\tau_k}\rangle 
+\langle \ell(t_k) - \ell(t_{k-1}),\tfrac{z_k - 
z_{k-1}}{\tau_k}\rangle
\\
&=:T_0 + T_1 + T_2.
\label{eq.discrincl03}
\end{align}
Clearly, $\abs{T_0}\leq 
\dist_\bbV(-\rmD\calE(0,z_0),\partial\calR(0))\norm{(z_1-z_0)/\tau_1}_\bbV$.  
With \eqref{est:DF} and \eqref{eq.Mief100}, the  term $T_1$  is estimated as
\begin{align*}
 \abs{T_1}\leq \tfrac{\alpha}{2}\sum_{k=0}^N \tfrac{1}{\tau_k}
 \norm{z_k -z_{k-1}}_\calZ^2 + C_\alpha\sum_{k=0}^N\tfrac{1}{\tau_k}
 \norm{z_k -z_{k-1}}_\calV^2.
\end{align*}
In the  term $T_2$ we shift once more the indices and obtain
\begin{multline*}
 \abs{T_2}
 \leq \abs{\langle \ell(t_N),\tfrac{z_N - z_{N-1}}{\tau_N}\rangle} + 
  \abs{\langle \ell(t_1),\tfrac{z_1 - z_{0}}{\tau_1}\rangle} + 
\sum_{k=1}^{N-1}\abs{\langle \ell(t_k),\tfrac{z_k-z_{k-1}}{\tau_k} 
-\tfrac{z_{k+1}-z_{k}}{\tau_{k+1}}\rangle}   
\\
\leq \tfrac{\varepsilon}{4}\left(\norm{\tfrac{z_N - z_{N-1}}{\tau_N}}_\bbV^2 + 
\norm{\tfrac{z_1 - z_{0}}{\tau_1}}_\bbV^2 + 
\sum_{k=1}^N\norm{\Big(\tfrac{z_k - 
z_{k-1}}{\tau_{k}}\Big) 
- \Big(\tfrac{z_{k-1} - z_{k-2}}{\tau_{k-1}}\Big)}_\bbV^2\right)
+ C_\varepsilon\bnorm{ \ell}_{L^\infty(0,T;\calV^*)}^2,
\end{multline*}
where in the last line we applied the Young inequality. Inserting these 
estimates into \eqref{eq.discrincl03}, rearranging the terms and neglecting 
some nonnegative terms on the left hand side we finally arrive at
\begin{multline*}
 \tfrac{\alpha}{2}\sum_{k=0}^N \tau_k\norm{\tfrac{z_{k-1} - 
z_{k-2}}{\tau_k}}^2_\calZ
\\
\leq C_\varepsilon
\left(\dist_\bbV(-\rmD\calE(0,z_0);\partial\calR(0))+\norm{\ell}_{
L^\infty((0 , T);\calV^*) }
\right)^2 + 
C_\alpha\sum_{k=0}^N\tfrac{1}{\tau_k}
 \norm{z_k -z_{k-1}}_\calV^2. 
\end{multline*}
By \eqref{est.unifdissip1}, the last term on the right hand side is bounded 
by $C\varepsilon^{-1}$, uniformly in~$N$. This proves~\eqref{est.bv04}.
\end{proof}

\section{Existence and uniqueness of viscous solutions}
\label{s-viscousE}
The  aim of this section is to prove the existence of solutions to the 
following system
for $\varepsilon>0$ and given initial value $z_0\in \calZ$:
\begin{align}
\label{eq.viscincl01}
 0\in \partial\calR_\varepsilon(\dot z(t)) + 
\rmD\calE(t,z(t)),\quad z(0)=z_0\,.
\end{align}
\begin{definition}
\label{def.viscsol}
 Let $\varepsilon>0$, $\ell\in BV([0,T];\calV^*)$, $z_0\in \calZ$.  A function 
$z\in H^1([0,T];\calV)$ $\cap L^\infty((0,T);\calZ)$ is a weak solution to 
\eqref{eq.viscincl01} if $z(0)=z_0$ and if the inclusion \eqref{eq.viscincl01}  
is satisfied for 
almost all $t\in [0,T]$.  
\end{definition}
As is common in the study of rate independent systems it is more convenient to 
work with an equivalent formulation, namely De Giorgi's energy dissipation 
principle.
\begin{lemma}\label{lem:EDP}
 Let  $z\in H^1([0,T];\calV)\cap L^\infty((0,T);\calZ)$ with 
$z(0)=z_0\in \calZ$. The following properties are equivalent: 
\begin{enumerate}
\item[(a)] $z$
is a weak solution to 
\eqref{eq.viscincl01} in the sense of Definition \ref{def.viscsol}.
\item[(b)] For all $t\in [0,T]$ we have
\begin{align}
 \label{eq.EDP01}
 \calI(z(t)) +\int_0^t\calR_\varepsilon(\dot z(s)) + 
\calR_\varepsilon^*(-\rmD\calE(s,z(s)))\ds =  \calI(z_0) + 
\int_0^t\langle\ell(s),\dot z(s)\rangle \ds. 
\end{align}
\item[(c)] For all $t\in [0,T]$ we have
\begin{align}
 \label{eq.EDP02}
 \calI(z(t)) +\int_0^t\calR_\varepsilon(\dot z(s)) + 
\calR_\varepsilon^*(-\rmD\calE(s,z(s)))\ds \leq   \calI(z_0) + 
\int_0^t\langle\ell(s),\dot z(s)\rangle \ds. 
\end{align}
\end{enumerate}
If $z$ satisfies any of these properties then $Az\in 
L^\infty((0,T);\calZ^*)\cap L^2((0,T);\calV^*)$. 
\end{lemma}
\begin{proof}
 The proof follows standard arguments relying on convex analysis and the  
chain rules provided in Proposition \ref{prop.chainrules}, 
see e.g.\ \cite[Proposition E.1]{KneesThomas2018}.

Indeed, let $z$ be a weak solution to \eqref{eq.viscincl01}.  The fact that  
$\partial\calR(0)$ can be identified with a subset of 
$\calV^*$ that is bounded with respect to the norm in $\calV^*$, and the 
assumptions on $\calF$ and $\ell$ imply that  $Az\in 
L^\infty((0,T);\calZ^*)\cap L^2((0,T);\calV^*)$. Convex analysis 
arguments and the chain rule provided in Proposition \ref{prop.chainrules}  
yield the identity 
\begin{multline*}
 \calR_\varepsilon(\dot z(t)) + \calR_\varepsilon^*(-\rmD\calE(t,z(t)))
 \\
 = 
 \langle -\rmD\calI(z(t)),\dot z(t)\rangle_{\calV^*,\calV} 
 +\langle 
\ell(t),\dot z(t)\rangle_{\calV^*,\calV}
=-\tfrac{\rmd}{\rmd t}\calI(z(t))
+ \langle 
\ell(t),\dot z(t)\rangle_{\calV^*,\calV}
\end{multline*}
that is valid for almost all $t$. Integration with respect to $t$ implies 
\eqref{eq.EDP01}. From this, \eqref{eq.EDP02} is an obvious consequence. 

Assume now that $z$ satisfies \eqref{eq.EDP02}. Since 
$\int_0^T\calR^*_\varepsilon(-\rmD\calE(r,z(r)))\dr<\infty$, it follows that 
$\rmD\calE(\cdot,z(\cdot))\in L^2(0,T;\calV^*)$ and in particular that $Az\in 
L^\infty((0,T);\calZ^*)\cap L^2((0,T);\calV^*)$. By the Fenchel 
inequality and the chain rule we deduce
\begin{multline*}
 \int_0^t\langle -\rmD\calE(s,z(s)),\dot z(s)\rangle \ds
 \leq \int_0^t \calR_\varepsilon(\dot z(s)) + 
\calR_\varepsilon^*(-\rmD\calE(s,z(s)))\ds
\\
\overset{\eqref{eq.EDP02}}{\leq}
\calI(z_0) - \calI(z(t)) + \int_0^t\langle \ell(s),\dot z(s)\rangle\ds
=\int_0^t\big(-\frac{\rmd}{\rmd t}\calI(z(s)) \big)+\langle \ell(s),\dot 
z(s)\rangle\ds.
\end{multline*}
Hence, \eqref{eq.EDP01} is valid. Localizing the integral identity and using 
once more the tools from convex  analysis finally shows that $z$ is a weak 
solution. 
\end{proof}

 For $\ell\in BV([0,T];\calV^*)$ let $\ell_-$ and $\ell_+$ denote the left and 
the right continuous representative. The identity \eqref{eq.EDP01} reveals that 
the weak solutions of \eqref{eq.viscincl01} for  $\ell$ are also weak 
solutions for  $\ell_+$ and $\ell_-$.

\begin{proposition}
\label{prop.unifboundeps}
 Assume \eqref{eq.Mief000}--\eqref{eq.Mief100}. 
 For every $\ell\in BV([0,T];\calV^*)$, $z_0\in \calZ$ and $\varepsilon>0$ 
there exists a unique weak solution $z_\varepsilon$ of \eqref{eq.viscincl01}. 
This solution coincides with the weak solutions for $\ell_+$ and $\ell_-$. 
Moreover, $\sup_{\varepsilon>0} 
\norm{z_\varepsilon}_{L^\infty((0,T);\calZ)}<\infty$. 

If in addition we assume that $\rmD\calE(0,z_0)\in \calV^*$,  
then the 
weak solution belongs to 
$H^1((0,T);\calZ)$ and 
there exists a constant $C>0$ such that for all $\varepsilon>0$  the 
corresponding  weak solution satisfies 
\begin{align}
\label{est.bv03}  
\norm{z_\varepsilon}_{W^{1,1}((0,T);\calZ)} 
 + \varepsilon\norm{\dot z_\varepsilon}_{L^\infty(0,T;\calV)}
 + \norm{\rmD\calE( \cdot, z_\varepsilon)}_{L^\infty((0,T);\calV^*)}\leq C. 
\end{align}
\end{proposition}
\begin{remark}
\label{rem.viscindepT}
The constant in 
  \eqref{est.bv03} has the same structure as the constants in 
  \eqref{est.bv01}--\eqref{est.bv02}. 
\end{remark}

\begin{proof}[Proof of Proposition \ref{prop.unifboundeps}]
Uniqueness of weak solutions: 

For $i\in \{1,2\}$ let $\ell_i\in \{\ell,\ell_+,\ell_-\}$ and let $z_i$ be a 
weak solution for \eqref{eq.viscincl01} corresponding to  $\ell_i$ with 
$z_i(0)=z_0$. Since $\partial\calR$ is maximal monotone, the inclusion  
\eqref{eq.viscincl01} implies
\begin{multline*}
 \langle A( z_1(t)- z_2(t)),\dot z_1(t) - \dot z_2(t)\rangle_{\calV^*,\calV} + 
\varepsilon\norm{\dot z_1(t) -\dot z_2(t)}^2_\calV 
\\\leq \langle 
\rmD\calF(z_2(t))-\rmD\calF(z_1(t)) 
+ ( \ell_1(t) - \ell_2(t)) , 
\dot z_1(t) - \dot z_2(t)\rangle_{\calV^*,\calV}, 
\end{multline*}
which is valid for almost all $t\in [0,T]$. Integration with respect to $t$ 
yields
\begin{multline*}
 \tfrac{\alpha}{2}\norm{z_1(t) - z_2(t)}_{\calZ}^2 
+\varepsilon\int_0^t\norm{\dot z_1(s) -\dot  z_2(s)}^2_\calV\ds 
\\
\leq
\tfrac{\alpha}{2}\norm{z_1(0) - z_2(0)}_{\calZ}^2 
+\int_0^t  \langle \rmD\calF(z_2(s)) 
-\rmD\calF(z_1(s)),\dot z_1(s) -\dot z_2(s)\rangle\ds 
\\+ \int_0^t\norm{\ell_1(s) - \ell_2(s)}_{\calV^*}
\norm{\dot z_1(s) -\dot z_2(s)}_\calV\ds.
\end{multline*}
Observe that the first and the last term on the right hand side are zero since 
$\ell_1$ and $\ell_2$ differ at most on a countable set.  
Thanks to \eqref{ass.F01} and Young's inequality the integral on the right hand 
side can be estimated as
\begin{align*}
 \int_0^t\langle \rmD\calF(z_2(s)) 
& -\rmD\calF(z_1(s)) ,\dot z_1(s) -\dot z_2(s)\rangle\ds 
 \\
&\leq C\int_0^t \norm{z_1(s)-z_2(s)}_\calZ
\norm{\dot z_1(s) - \dot z_2(s)}_\calV\ds\\
&\leq \int_0^t\tfrac{\varepsilon}{2}\norm{\dot z_1(s) -\dot  z_2(s)}^2_\calV\ds 
+ 
C_\varepsilon\int_0^t\norm{z_1(s) - z_2(s)}_\calZ^2\ds.
\end{align*}
Joining these inequalities and applying the Gronwall Lemma finishes the proof 
of uniqueness.

Existence of weak solutions: 

 Let $\varepsilon>0$ be fixed. Let $(\Pi_N)_{N\in \N}$ be a sequence of 
partitions of $[0,T]$ with fineness $\triangle_N\searrow 0$ and let 
$(z_k^N)_{k\leq N}$ 
be minimizers of \eqref{eq.time-incr}. We introduce the following piecewise 
affine and piecewise linear interpolants: 
\begin{gather*}
 \wt z_N(t):=z_{k-1}^N + \tfrac{t-t_{k-1}}{\tau_k}(z^N_k - z_{k-1}^N),\quad 
t\in [t_{k-1}^N,t_k^N],\\
\underline{z}_N(t):=z_{k-1}^N, \quad 
t\in [t_{k-1}^N,t_k^N); \,\,\,\,\, 
\overline z_N(t):= z_k^N, \,\,  
\bar t_N(t):=t_k^N, \quad 
t\in (t_{k-1}^N,t_k^N].
\end{gather*}
By Proposition \ref{prop.basicest} the functions $\wt z_N$, $\bar z_N$, 
$\underline{z}_N$ are uniformly bounded (w.r.\ to $N$ and $\varepsilon$) in the 
space $L^\infty((0,T);\calZ)$. Moreover, we have 
\begin{align}        
 \label{est.prvisc01}
 \norm{\wt z_N}_{H^1((0,T);\calV)}\leq C/\sqrt{\varepsilon}
\end{align}
 with a 
constant $C>0$ that is 
independent of the partition $\Pi_N$. Thus, there exists $z\in 
L^\infty((0,T);\calZ)\cap H^1((0,T);\calV)$ and a (not relabeled) subsequence 
such that
\begin{align}
 \wt z_N,\overline z_N,\underline z_N&\overset{*}{\rightharpoonup} z \text{ 
weakly$*$ in } L^\infty((0,T);\calZ),
\label{eq.convvisc01}\\
\wt z_N&\rightharpoonup z \text{ weakly in }H^1((0,T);\calV),
\label{eq.convvisc02}\\
\wt z_N(t),\overline z_N(t),\underline{z}_N(t)&\rightharpoonup z(t) \text{ 
weakly in  $\calZ$ for all $t\in [0,T]$}, 
\label{eq.convvisc03}
\end{align}
where the last line is a consequence of \eqref{eq.convvisc01} and 
\eqref{eq.convvisc02}. Thanks to \eqref{est.prvisc01} the limits of the 
different interpolants coincide.  
All accumulation points obtained in this way are uniformly bounded 
in $L^\infty((0,T);\calZ)$ with respect to  
$\varepsilon>0$ and the chosen sequence of  partitions.  
With the above definitions, for $t>0$ the inclusion  
\eqref{eq.discrincl01} can be rewritten as 
$
 -\rmD\calE(\bar t_N(t),\overline z_N(t))\in 
\partial\calR_\varepsilon(\dot{\wt z}_N(t)), 
$ 
and by convex analysis and the chain rule we obtain 
\begin{multline*}
 \calR_\varepsilon(\dot{\wt z}_N(t)) + \calR_\varepsilon^*(-\rmD\calE(\bar 
t_N(t),\overline z_N(t)))
\\= -\tfrac{\rmd}{\rmd t} \calI(\wt z(t)) 
+ \langle \ell(\bar t_N(t)),\dot{\wt z}_N(t)\rangle
+ \langle \rmD\calI(\wt z(t)) - \rmD\calI(\overline z(t)),\dot{\wt 
z}_N(t)\rangle.
\end{multline*}
Integration with respect to $t$ results in a  discrete version of the energy 
dissipation estimate \eqref{eq.EDP02} with an additional error term: For all 
$t\in [0,T]$
\begin{multline}
\label{est.discrviscEDI01}
 \calI(\wt z_N(t)) +\int_0^t \calR_\varepsilon(\dot{\wt z}_N(s)) + 
\calR_\varepsilon^*(-\rmD\calE(\bar t_N(s),\overline z_N(s)))\ds
\\
\leq 
\calI(z_0) + \int_0^t \langle \ell(\bar t_N(s)),\dot{\wt z}_N(s)\rangle\ds 
+\int_0^t r_N(s)\ds, 
\end{multline}
where  $r_N(t)= \langle \rmD\calI(\wt z_N(t)) - \rmD\calI(\overline z_N(t)),
\dot{\wt z}_N(t)\rangle$. Next we pass to the limit $N\to \infty$ in 
\eqref{est.discrviscEDI01}. 
Since $\wt z_N(t) - \overline z_N(t) =\dot{\wt z}_N(t)(t-\bar t_N(t))$, 
with \eqref{est.lambda-convex} we find
\begin{align*}
 r_N(t)&=-(\bar t_N(t)-t)^{-1} 
 \langle \rmD\calI(\wt z_N(t)) -\rmD \calI(\overline z_N(t)),\wt z_N(t) - 
\overline 
z_N(t)\rangle \leq \lambda \tau_k\bnorm{\dot{\wt z}_N(t)}_\calV^2, 
\end{align*}
and $\lambda>0$ is independent of $\varepsilon>0$ and the  partition $\Pi_N$. 
Hence, relying on estimate \eqref{est.prvisc01} we obtain  
\begin{align*}
 \limsup_{N\to \infty}\int_0^t r_N(s)\ds \leq \lambda \limsup_{N\to\infty}  
\triangle_N \norm{\wt z_N}^2_{H^1((0,T);\calV)}=0,
\end{align*}
as $ \limsup_{N\to \infty}\triangle_N =0$. 
Concerning the power term observe first that $\bar t_N(t)\searrow t$ for $N\to 
\infty$, and hence, $\ell(\bar t_N(t))\to \ell(t+)=\ell_+(t)$ strongly in 
$\calV^*$ (for 
all $t\in [0,T]$). Since $\ell\in L^\infty((0,T);\calV^*)$ this implies in 
particular that $\ell\circ \bar t_N\to \ell_+$ strongly in 
$L^2((0,T);\calV^*)$. Taking into account  the weak convergence of 
$(\dot{\wt z}_N)_N$  in $L^2((0,T);\calV)$ we obtain
\begin{align*}
 \int_0^t\langle \ell\circ \bar t_N,\dot{\wt z}_N\rangle\ds
 \to 
\int_0^t\langle \ell_+,\dot z\rangle\ds. 
\end{align*}
The discrete energy dissipation estimate \eqref{est.discrviscEDI01} in 
particular implies that 
\[
\sup_N\int_0^T\calR_\varepsilon^*(-\rmD\calE(\bar 
t_N,\overline z_N))\ds <\infty
\]
and hence $\rmD\calE(\bar t_N,\overline z_N)$ is 
uniformly bounded (with respect to $N$) in $L^2((0,T);\calV^*)$. Thanks to 
\eqref{eq.convvisc03} we also have pointwise weak convergence in $\calZ^*$ of 
$\rmD\calE(\bar t_N(t),\overline z_N(t))$ to $\rmD\calE(t+,z(t))$ so that 
altogether $\rmD\calE(\bar t_N,\overline z_N)\rightharpoonup  \rmD\calE(\cdot 
+ 
, z(\cdot))$ weakly in $L^2((0,T);\calV^*)$. By lower semicontinuity we 
therefore obtain for the left hand side in \eqref{est.discrviscEDI01}
\begin{align*}
 \liminf_{N}\text{(l.h.s)}\geq \calI( z(t)) +\int_0^t 
\calR_\varepsilon(\dot{z}(s)) + 
\calR_\varepsilon^*(-\rmD\calE(s+, z(s)))\ds. 
\end{align*}
In summary we have shown that $z$ satisfies \eqref{eq.EDP02} with $\ell_+$ and 
therefore also with $\ell$. Hence, by Lemma \ref{lem:EDP} $z$ is a weak 
solution to 
\eqref{eq.viscincl01} for $\ell$. 

Improved estimates: 
Assume in addition that $\rmD\calE(0,z_0)\in\calV^*$. Then 
from 
Proposition \ref{prop.BV-estimate} we obtain 
\begin{align*}
 \norm{\wt z_N}_{W^{1,1}((0,T);\calZ)} 
 + \varepsilon\norm{\dot{\wt z}_N}_{L^\infty(0,T;\calV)}
 + \norm{\rmD\calE(\bar t_N,\overline z_N)}_{L^\infty((0,T);\calV^*)}\leq C, 
\end{align*}
and $C>0$ is independent of $\varepsilon$ and $\Pi_N$. 
Moreover, $\norm{\wt z_N}_{H^1((0,T),\calZ)}\leq C_\varepsilon$, uniformly in 
$N$.  
Hence, by weak compactness and  lower semicontinuity, 
for $N\to\infty$ we obtain the improved regularity of $z$ as well as  
 \eqref{est.bv03}. 
\end{proof}
\section{The viscosity limit} 
\label{s-epsto0}
In order to study the limit $\varepsilon\to 0$ we use the reparameterization 
technique originally introduced in \cite{MiEf06} and refined in \cite{MRS16}, 
among others. 
In this section we assume 
\begin{align}
\label{eq.assumptionvv1}
 \text{\eqref{eq.Mief000}--\eqref{eq.Mief100} and that 
$\rmD\calE(0,z_0)\in\calV^*$.
}
\end{align}

Let 
\[ 
     \vvp:\calV\times\calV^*\to\R,\quad \vvp(v,w):=\calR(v) 
+\norm{v}_\bbV\dist_\bbV(w,\partial\calR(0))
\]
denote the so called 
\textit{vanishing viscosity contact potential}, \cite{MRS16}. Observe that by 
Young's inequality, for all $\varepsilon>0$ we have 
$\vvp(v,w)\leq \calR_\varepsilon(v) + \calR_\varepsilon^*(w)$. Let 
$\varepsilon>0$ and let $z_\varepsilon$ be a weak solution of the viscous 
problem \eqref{eq.viscincl01}. As in \cite{MRS16}, 
we define 
\begin{align}
 s_\varepsilon(t):= t +\int_0^t \vvp(\dot 
z_\varepsilon(r),-\rmD\calE(r,z_\varepsilon(r))\dr,\quad 
S_\varepsilon:=s_\varepsilon(T)\,.
\label{s_eps}
\end{align}
By definition, $s_\varepsilon:[0,T]\to[0,S_\varepsilon]$ is strictly monotone 
and hence invertible. We denote with $\hat 
t_\varepsilon:[0,S_\varepsilon]\to[0,T]$ the inverse of $s_\varepsilon$.  
Furthermore, let
\begin{align}
 \hat z_\varepsilon(s):=z_\varepsilon(\hat t_\varepsilon(s)),\quad 
 \hat \ell_\varepsilon(s):=\ell(\hat t_\varepsilon(s))\,.
\end{align}
Clearly, $\hat t_\varepsilon\in W^{1,\infty}((0,S_\varepsilon))$ and for almost 
all $s$ we have 
\begin{align}
\label{eq.normalizationeps}
\hat t_\varepsilon'(s) + \vvp(\hat 
z_\varepsilon'(s),-\rmD\calE(\hat t_\varepsilon(s),\hat z_\varepsilon(s)))=1.
\end{align}
In the next proposition we collect regularity properties and (uniform) 
estimates that are valid for the transformed quantities.
\begin{proposition}
\label{prop.unifboundeps-param}
 Assume \eqref{eq.assumptionvv1}. 
 Then
 $\sup_{\varepsilon>0} S_\varepsilon<\infty$, 
 $\hat z_\varepsilon$ belongs to the space $ H^1((0,S_\varepsilon);\calZ)\cap 
W^{1,\infty}((0,S_\varepsilon);\calV)$ and there is a constant $C>0$ such that 
for all $\varepsilon>0$ and with $I_\varepsilon:=(0,S_\varepsilon)$ we have 
\begin{align}
\label{est.ztrafo1}
 \norm{\hat 
z_\varepsilon}_{W^{1,1}(I_\varepsilon;\calZ)} + 
\norm{\hat z_\varepsilon'}_{L^\infty(I_\varepsilon;\calX)} 
 +\varepsilon\norm{(\hat t_\varepsilon')^{-1}\hat 
z_\varepsilon'}_{L^\infty(I_\varepsilon;\calV)} 
+\norm{\rmD\calE(\hat t_\varepsilon,\hat 
z_\varepsilon)}_{L^\infty(I_\varepsilon;\calV^*)} <C.
\end{align}
Moreover, $\hat\ell_\varepsilon\in BV([0,S_\varepsilon];\calV^*)$ with
$\Var_{\calV^*}(\hat \ell_\varepsilon,[0,S_\varepsilon]) 
= \Var_{\calV^*}(\ell,[0,T])$.
\end{proposition}

\begin{proof}
 Observe that $S_\varepsilon\leq T + \int_0^T \calR_\varepsilon(\dot 
z_\varepsilon(r)) +\calR_\varepsilon^*(-\rmD\calE(r,z_\varepsilon(r)))\dr$. 
From 
the identity \eqref{eq.EDP01} and estimate \eqref{est.bv03} we deduce the 
uniform bound for $(S_\varepsilon)_\varepsilon$. Since $\hat t_\varepsilon$ is 
Lipschitz continuous, the regularity of $\hat z_\varepsilon$ and  estimate 
\eqref{est.ztrafo1}  immediately follow from  Proposition 
\ref{prop.unifboundeps}. 
 Observe finally that thanks to the strict monotonicity of $s_\varepsilon$ we 
have  $\Var_{\calV^*}(\hat \ell_\varepsilon,[a,b])= \Var_{\calV^*}(\ell,[\hat 
t_\varepsilon(a),\hat t_\varepsilon(b)])$. 
\end{proof}
As a consequence, by compactness we obtain 

\begin{proposition}
\label{prop.param-conv}
Assume \eqref{eq.assumptionvv1}. 
 
 Let $(\varepsilon_n)_{n\in \N}$ be a sequence with $\varepsilon_n\searrow 0$ 
for $n\to \infty$. Then there exist $S>0$,  a triple $(\hat t,\hat 
z,\hat\ell)$ with $\hat t\in W^{1,\infty}(0,S;\R)$, $\hat z\in  
AC^\infty([0,S];\calX)
\cap C([0,S];\calV)\cap
BV([0,S];\calZ)\cap  C_\text{weak}([0,S];\calZ)$ and 
$\hat \ell\in  BV([0,S];\calV^*)$ 
and a subsequence of $(\varepsilon_n)_n$ such that for $n\to \infty$ (we 
suppress the index $n$)
\begin{align}
 S_\varepsilon&\to S;\quad \hat t_\varepsilon\overset{*}{\rightharpoonup}\hat t 
\text{ weakly$*$ in }W^{1,\infty}(0,S),\,\, \hat t(S)=T,\\
\hat z_\varepsilon& \rightharpoonup\hat z \text{ weakly$*$ in 
}L^\infty(0,S;\calZ) \text{ and uniformly in }C([0,S];\calV), 
\\
\hat\ell_\varepsilon&\overset{*}{\rightharpoonup}\hat\ell,\,\, 
\rmD\calI(\hat z_\varepsilon)\overset{*}{\rightharpoonup}\rmD\calI(\hat z) 
\text{ weakly$*$ in }L^\infty(0,S;\calV^*),  
\label{est.convparam3}
\end{align}
and for every $s\in [0,S]$
\begin{align}
 \hat t_\varepsilon(s)&\to\hat t(s), 
 \quad 
 \hat z_\varepsilon(s)\rightharpoonup\hat z(s) \text{ weakly in $\calZ$ }, 
 \label{est.convparam4}\\
 \rmD\calI(\hat z_\varepsilon(s))&\rightharpoonup\rmD\calI(\hat z(s)) \text{ 
weakly in $\calV^*$},
\quad 
\hat\ell_\varepsilon(s)\rightarrow\hat\ell(s) \text{ strongly in }\calV^*. 
\label{est.convparam5}
\end{align} 
The function $s\mapsto\calI(\hat z(s))$ is uniformly continuous on $[0,S]$, the 
function  $s\mapsto \rmD\calI(\hat z(s))$ belongs to 
$C_\text{weak}([0,S];\calV^*)$ and   
  $\hat t'(s)\geq 0$ for almost all $s$. Moreover,  $\hat \ell$ 
can be characterized as follows: 
For every $t_*\in [0,T]$ there exists $s_*\in \hat t^{-1}(t_*)$ such that for 
all $s\in [0,S]$ with $\hat t(s)=t_*$ we have
\begin{align}
\label{eq.charhatell}
 \hat \ell(s)=\begin{cases}
               \ell(\hat t(s)-)&s<s_*\\
               \ell(\hat t(s)+) & s>s_*
              \end{cases}\, \quad\text{and }\quad\hat \ell(s_*)\in 
\{\ell(t_*),\ell(t_*+),\ell(t_*-)\}.
\end{align}
\end{proposition}

\begin{remark}
 In the previous proposition we tacitly extend all functions by their constant 
value in $S_\varepsilon$, if $S_\varepsilon<S$. 
\end{remark}

\begin{proof}[Proof of Proposition \ref{prop.param-conv}]
The uniform bounds provided in Proposition \ref{prop.unifboundeps-param} in 
combination with Proposition \ref{prop.hellyarzasc-version2} yield the 
convergence 
 properties of the sequence  $(\hat z_\varepsilon)_\varepsilon$ and the 
regularity of the limit function $\hat z$.  
 The first assertion in \eqref{est.convparam5} is a consequence of the weak 
continuity of $\rmD\calI:\calZ\to \calZ^*$,  \eqref{est.convparam4} and the 
uniform estimate \eqref{est.ztrafo1}. From this we also obtain the second part 
of  \eqref{est.convparam3}.  By the very same argument the weak continuity of 
$s\mapsto \rmD\calI(\hat z(s))$ in $\calV^*$ ensues.

Let us next show  that $s\mapsto\calI(\hat z(s))$ is  continuous 
and thus uniformly 
continuous on $[0,S]$. As stated above, we have 
$\rmD\calI(\hat z(\cdot))\in C_\text{weak}([0,S];\calV^*)$. But this is 
also separately valid for  the mappings $s\mapsto A\hat z(s)$ 
and $z\mapsto 
\rmD\calF(\hat z(s))$. Indeed, since $\hat z\in L^\infty(0,S;\calZ)$ the 
assumed bound in   \eqref{ass.F01} yields $\rmD\calF(\hat z(\cdot))\in 
L^\infty(0,S;\calV^*)$. Combining this with assumption \eqref{ass.fweakconv} 
and 
the fact that  $\hat z\in C_\text{weak}([0,S];\calZ)$, we  obtain 
$\rmD\calF(\hat z(\cdot))\in C_\text{weak}([0,S];\calV^*)$, and hence also 
$A\hat z(\cdot)\in C_\text{weak}([0,S];\calV^*)$. By standard arguments 
we ultimately obtain the continuity of $s\mapsto \calI(\hat z(s))$. 

It remains to discuss the sequence $(\hat\ell_\varepsilon)_\varepsilon$. The 
Banach space valued version of  
Helly's selection principle, \cite{BarbuPrecupanu86}, applied to  the 
sequence $(\hat 
\ell_\varepsilon)_\varepsilon$ yields \eqref{est.convparam3} and weak 
convergence in  \eqref{est.convparam5}. 
Since $\ell$ 
possesses (strong) left and right limits in $\calV^*$ and since $ (\ell(\hat 
t_\varepsilon(s)))_\varepsilon$ converges weakly for all $s$, it follows that 
$\hat\ell(s)$ belongs to the set $\{\ell(\hat t(s)),\ell(\hat 
t(s)+),\ell(\hat t(s)-)\}$ and that 
$\hat\ell_\varepsilon(s)\to \hat \ell(s)$ strongly in $\calV^*$. 
Let $t_*\in [0,T]$. If $t_*$ is a point of continuity of $\ell$, the proof of 
the representation formula for $\hat \ell$ is finished. Assume now that $t_*$ 
is a jump point of $\ell$ with $\ell(t_*-)\neq \ell(t_*+)$ (the arguments here 
below can easily be adapted to the case $\ell(t_*-)=\ell(t_*+)\neq\ell(t_*)$).  
By monotonicity and continuity of $\hat t$ we have  $\hat t^{-1}(t_*)=[a,b]$ 
for some $a<b$. Let $s\in 
[a,b]$ with $\hat\ell(s)=\ell(t_*+)$. This implies that there is 
$\varepsilon_0>0$ such that for all $\varepsilon<\varepsilon_0$ we have $\hat 
t_\varepsilon(s)\geq t_*$.  Again by monotonicity this implies that $\hat 
t_\varepsilon(\sigma)\geq t_*$ for every $\sigma\in [s,b]$  
and every $\varepsilon<\varepsilon_0$. Hence, for all these $\sigma$ we have 
$\hat\ell(\sigma)=\ell(t_*+)$. Let $s_+:=\inf\Set{s\in 
[a,b]}{\hat\ell(s)=\ell(t_*+)}$. Then $\hat \ell(s)=\ell(t_*+)$ for all $s\in 
(s_+,b]$. In a similar way we define 
$s_-:=\sup\Set{s\in [a,b]}{\hat\ell(s)=\ell(t_*-)}$ and obtain 
$\hat\ell(s)=\ell(t_*-)$ for all $s\in[a,s_-)$. Observe that $s_-\leq s_+$. 
Assume now that $s_-<s_+$ and let $s_1<s_2\in (s_-,s_+)$ which implies  
$\hat  \ell(s_1)=\hat \ell(s_2)=\ell(t_*)$. But this is only 
possible if   there exists 
$\varepsilon_1>0$ such that for all $\varepsilon<\varepsilon_1$ we have $\hat 
t_\varepsilon(s_1)=t_*=\hat t_\varepsilon(s_2)$, which  is a 
contradiction to 
the strict monotonicity of $\hat t_\varepsilon$. Hence, $s_-=s_+=:s_*$ and the 
proof is finished.  
\end{proof}

Next we rewrite the energy dissipation estimate \eqref{eq.EDP02} in the new 
variables and investigate the limit $\varepsilon\to 0$. For that purpose we 
need to introduce some more notation. For a curve $z:[0,S]\to\calX$ 
we 
define 
\begin{align*}
 \Var_\calR(z,[a,b]):=\sup_{\text{partitions $(t_i)_i$ of 
$[a,b]$}}\sum_{i=1}^m\calR(z(t_i)-z(t_{i-1}))\,
\end{align*}
as the $\calR$ dissipation ($\calR$ variation)  along the curve $z$. 
Thanks to the assumptions on $\calR$ we have $\Var_\calR(z;[a,b])<\infty$ if 
and only if $\Var_\calX(z;[a,b])<\infty$.

Let $\hat \calE(s,v):=\calI(v) - \langle\hat\ell(s),v\rangle$. In order to 
shorten the notation let 
\begin{align}
\label{def.abbr-e}
\vve(f,v):=\dist_{\bbV}(-\rmD\calI(v) +f,\partial\calR(0)).
\end{align}
With this, 
$\dist_\bbV(-\rmD\hat\calE(s,\hat 
z(s)),\partial\calR(0))=\vve(\hat\ell(s),\hat 
z(s))$. For $f\in BV([0,S];\calV^*)$ and $v\in \calZ$  let 
\begin{align}
\label{def.vvm}
 \vvm(f(s),v):=\min\{\vve(f(s),v), \vve(f(s-),v),\vve(f(s+),v)\}. 
\end{align}
The next lemma shows that $\vvm(\cdot,\cdot)$ is lower semicontinuous. 
\begin{lemma}
 \label{lem:vvmlsc}
 Let $f\in BV([0,S];\calV^*)$, $(v_n)_n\subset \calZ$ with 
$\rmD\calI(v_n)\rightharpoonup\rmD\calI(v)$ weakly in $\calV^*$ and 
$(s_n)_n,s\subset[0,S]$ with $s_n\to s$. Then $\liminf_n \vvm(f(s_n), v_n)\geq 
m(f(s),v)$. 
\end{lemma}
\begin{proof}
 Observe that the accumulation points of the sequences 
$(f(s_n+))_n$, $(f(s_n))_n$, $(f(s_n-))_n$ belong to the set $\{f(s), f(s+), 
f(s-)\}$. Hence, by the lower semicontinuity of the functional $\dist_\bbV$ we 
conclude.
\end{proof}

\begin{theorem}
 \label{ex.ri.pparam}
 Assume \eqref{eq.assumptionvv1}. 
 Then there exist $S>0$, $\hat t\in W^{1,\infty}(0,S;\R)$, $\hat z\in 
AC^\infty([0,S];\calX)
\cap C([0,S];\calV)\cap
BV([0,S];\calZ)\cap  C_\text{weak}([0,S];\calZ)$ and 
 $\hat \ell\in  
BV([0,S];\calV^*)$ as in \eqref{eq.charhatell}  
 such that $\calI(\hat z)\in C([0,S])$, $\rmD\calI(\hat z)\in 
L^\infty(0,S;\calV^*)\cap 
C_\text{weak}([0,S];\calV^*)$. 
Let  
$G:=\Set{s\in [0,S]}{\vvm(\hat \ell(s),\hat z(s))>0 }$. The set $G$ is 
open and 
 $\hat z\in W^{1,\infty}_\text{loc}(G;\calV)$. Moreover,  
for almost every $s\in [0,S]$ 
\begin{align}
\hat t'(s)\geq 0,\,\, \hat t(S)=T,\,\, \hat z(0)=z_0,
\label{eq.pparamsol-con1}
\\
  \hat t'(s)\dist_\bbV(-\rmD\hat \calE(s,\hat z(s)),\partial\calR(0))=0\,,
\label{eq.complementarity}
\\
1= \begin{cases}
 \hat t'(s) +\calR[\hat z'](s) &\text{if }s\notin G\\
 \hat t'(s) +\calR[\hat z'](s) +\norm{\hat z'(s)}_\calV\dist_\bbV(-\rmD\hat 
\calE(s,\hat z(s)),\partial\calR(0)) &\text{if }s\in G
\end{cases}\,.
\label{eq.normalized.sol}
  \end{align}
Furthermore, $\hat t'=0$ almost everywhere on $G$. 
Finally, 
 for every $s\in [0,S]$ 
\begin{multline}
 \label{eq.energydissipidentitylimitpparam}
 \calI(\hat z(s)) + \int_0^s\calR[\hat z'](r) \dr + 
 \int_{(0,s)\cap G} \norm{\hat z'(r)}_\bbV\dist_\bbV(-\rmD\hat \calE(r,\hat 
z(s)),\partial\calR(0))\dr 
\\
= \calI(z_0) + \int_0^s\langle \hat\ell(r),\rmd \hat z(r)\rangle\,.
\end{multline}

Every  tuple $(S,\hat t,\hat z,\hat \ell)$ obtained as a limit as in 
Proposition \ref{prop.param-conv} satisfies the above conditions. 
\end{theorem}
The integral on the right hand side in 
\eqref{eq.energydissipidentitylimitpparam} is understood as a Kurzweil 
integral, see Appendix \ref{sec.Kurzweil}. 

\begin{proof}
For $\varepsilon>0$ let $z_\varepsilon$ be a solution to \eqref{eq.viscincl01} 
and let $(S_\varepsilon,\hat t_\varepsilon,\hat 
z_\varepsilon,\hat \ell_\varepsilon)_{\varepsilon>0}$ be a sequence constructed 
from $(z_\varepsilon)_{\varepsilon}$  that converges to $(S,\hat t,\hat z,\hat 
\ell)$ as stated Proposition \ref{prop.param-conv}. The aim is to show that 
$(S,\hat t,\hat z,\hat \ell)$ has the properties formulated in Theorem 
\ref{ex.ri.pparam}.

Complementarity identity \eqref{eq.complementarity}: 
Since $\partial\calR(\dot z_\varepsilon(t))\subset\partial\calR(0)$, from 
\eqref{eq.viscincl01} we deduce 
\begin{align}
\label{est.exproofpp1}
 -\rmD\calE(\hat t_\varepsilon(s),\hat z_\varepsilon(s))\in \partial\calR(0) + 
\tfrac{\varepsilon}{\hat t'_\varepsilon(s)}\bbV\hat z_\varepsilon'(s), 
\end{align}
which implies that $\dist_\bbV(-\rmD\calE(\hat t_\varepsilon(s),\hat 
z_\varepsilon(s)),\partial\calR(0))\leq \tfrac{\varepsilon}{\hat 
t'_\varepsilon(s)}\norm{\hat z_\varepsilon'(s)}_\bbV$.  
Since $\partial\calR(0)$ is bounded in $\calV^*$, by lower 
semicontinuity and in combination  with 
\eqref{est.ztrafo1} and \eqref{est.convparam5} it follows that 
$\rmD\hat \calE(\cdot,\hat z(\cdot))\in L^\infty((0,S);\calV^*)$. Moreover, 
since $\varepsilon\norm{\dot{z}_\varepsilon}^2_{L^2((0,T);\calV)}$ is uniformly 
bounded (cf.\ \eqref{eq.EDP02} and Proposition \ref{prop.unifboundeps}),  
 we obtain  
\[
\sup_\varepsilon\varepsilon\norm{(\hat t_\varepsilon')^{-\frac{1}{2}}
\hat z_\varepsilon'}_{L^2((0,S);\calV)}^2 =
\sup_\varepsilon\varepsilon\norm{\dot 
z_\varepsilon}_{L^2((0,T);\calV)}^2 =:C <\infty.
\]
Since $\hat 
t_\varepsilon'(s)\leq 1$, we therefore arrive at 
$\int_0^S\big(\hat t_\varepsilon'\dist_\bbV(-\rmD\calE(\hat 
t_\varepsilon,\hat z_\varepsilon),\partial\calR(0))\big)^2\ds
\leq \varepsilon C$.
Thanks to \eqref{est.convparam5}, for almost every $s$  we 
have $\liminf_\varepsilon \dist_\bbV(-\rmD\calE(\hat 
t_\varepsilon(s),\hat z_\varepsilon(s)),\partial\calR(0))\geq 
\dist_\bbV(-\rmD\hat \calE(s,  
\hat z(s)),\partial\calR(0))$. Hence, Proposition 
\ref{app_prop:lsc} implies 
\begin{multline*}
 0\geq \liminf_\varepsilon \int_0^S(\hat 
t_\varepsilon')^2\dist_\bbV(-\rmD\calE(\hat 
t_\varepsilon,\hat z_\varepsilon),\partial\calR(0))^2\ds
\\
\geq \int_0^S (\hat 
t'(s))^2 \dist_\bbV(-\rmD\hat \calE(s,\hat z(s)),\partial\calR(0))^2\ds\geq 0
\end{multline*}
from which \eqref{eq.complementarity} is an immediate consequence. 

Energy dissipation 
estimate \eqref{eq.energydissipidentitylimitpparam}, $\leq$: 
For every $\varepsilon>0$ 
and $s\in [0,S]$ we have the energy dissipation estimate
\begin{align}
 \calI(\hat z_\varepsilon(s)) + \int_0^s \vvp(\hat 
z_\varepsilon'(r),-\rmD\calE(\hat t_\varepsilon(r),\hat z_\varepsilon(r))\dr 
\leq \calI(z_0) + \int_0^s\langle \hat\ell_\varepsilon(r),\hat 
z_\varepsilon'(r)\rangle\dr,
\end{align}
which is a reparameterized version of \eqref{eq.EDP02} in combination 
with the estimate for $\vvp(\cdot,\cdot)$. 

Thanks to Proposition \ref{prop:kurz-conv} we have 
$\int_0^s\langle\hat\ell_\varepsilon,\hat z_\varepsilon'\rangle\dr \to 
\int_0^s\langle \hat\ell(r),\rmd \hat z(r)\rangle$, where the last term is to 
be interpreted as a Kurzweil integral. By lower semicontinuity, for every $s$  
it holds $\liminf_\varepsilon\calI(\hat z_\varepsilon(s))\geq \calI(\hat z(s))$ 
and it remains to pass to the limit inferior in the dissipation integral. Again 
by Helly, \cite[Theorem 3.2]{MaMi05}, we obtain
\[
 \liminf_{\varepsilon\to 0}\int_0^s \calR(\hat z_\varepsilon'(r))\dr \geq 
\Var_\calR(\hat z,[0,s])=\int_0^s\calR[\hat z'](r)\dr, 
\]
where for the last identity we have applied 
Lemma \ref{app.lembv1} with $p=\infty$.  

The remaining term $\int_0^s \norm{\hat z'_\varepsilon(r)}_\bbV 
\vve(\hat\ell_\varepsilon(r),\hat z_\varepsilon(r))\dr$  is 
more delicate and we follow the arguments in \cite{MRS16} exploiting in 
addition   the uniform bound $\rmD\calI(\hat z_\varepsilon)\in 
L^\infty((0,S_\varepsilon);\calV^*)$. We recall the definition of 
$\vvm(\cdot,\cdot)$ in \eqref{def.vvm}.  The set  
\begin{align*}
 G=\Set{s\in [0,S]}{\vvm(\hat \ell(s),\hat 
z(s))>0 }\,
\end{align*}
 is relatively open 
(w.r.\ to $[0,S]$). Indeed, let  
$(s_n)_n\subset[0,S]\backslash G$ with $s_n\to s$. By Proposition 
\ref{prop.param-conv} we have $\rmD\calI(\hat z(s_n))\rightharpoonup
\rmD\calI(\hat z(s))$ weakly in $\calV^*$. Hence, with Lemma \ref{lem:vvmlsc} 
we obtain $0=\liminf_n m(\hat \ell(s_n),\hat z(s_n))\geq m(\hat\ell(s),\hat 
z(s))=0$, consequently  $s\notin G$.  

Next, as in \cite{MRS16}, we derive an improved uniform regularity estimate for 
$(\hat 
z_\varepsilon)_\varepsilon$ that is valid on compact subsets of $G$ and that 
allows us to give a meaning to $\hat z'$ on $G$.  Let $K\subset G$ be compact. 
By lower semicontinuity it follows that $c:=\inf_K \vvm(\hat\ell(s),\hat 
z(s))$ is positive. Again by lower semicontinuity for every $s\in K$ it holds
\begin{align*}
 \liminf_\varepsilon \vve(\hat \ell_\varepsilon(s),\hat z_\varepsilon(s))\geq 
\vvm(\hat \ell(s),\hat z(s))\geq c.
\end{align*}
Hence, for every $s\in K$ there exists $\varepsilon_0>0$ such that for all 
$\varepsilon<\varepsilon_0$ we have $\vve(\hat\ell_\varepsilon(s),\hat 
z_\varepsilon(s))\geq c/2$. A proof by contradiction shows that $\varepsilon_0$ 
in fact can  be chosen independently of $s\in K$. 
From the 
normalization property 
\eqref{eq.normalizationeps} we therefore deduce that 
$\sup_{\varepsilon<\varepsilon_0}\norm{\hat 
z_\varepsilon'}_{L^\infty(K;\calV)}\leq 2/c$ and hence $(\hat 
z_\varepsilon)_\varepsilon$ converges weakly$*$ in 
$W^{1,\infty}(K;\calV)$ to $\hat 
z$. 
Now we are in the position to apply Proposition \ref{app_prop:lsc} to conclude 
that 
\begin{multline}
 \label{est.proof-ex01}
\liminf_\varepsilon\int_K
\norm{\hat z_\varepsilon'(s)}_\calV
\dist_\bbV(-\rmD\calE(\hat t_\varepsilon(s),\hat 
z_\varepsilon(s)),\partial\calR(0))\ds 
\\
\geq \int_K \norm{\hat z'(s)}_\calV \dist_\bbV(-\rmD\hat\calE(s,\hat 
z(s)),\partial\calR(0))\ds.
\end{multline}
In summary we have proved 
\eqref{eq.energydissipidentitylimitpparam} with $\leq $ instead of equality. 
By similar arguments we obtain 
\eqref{eq.normalized.sol} with $\geq$ instead of equality.

In order to prove that in fact an identity is valid in 
\eqref{eq.energydissipidentitylimitpparam} and \eqref{eq.normalized.sol} we 
follow ideas 
from \cite{MielkeRossiSavare_COCV12}. 
For $s\in [0,S]$ let $\mu(s)\in \partial\calR(0)$ with $\bnorm{-\rmD\hat 
\calE(s,\hat z(s)) -\mu(s)}_{\calV^*}=\dist_{\bbV}(-\rmD\hat \calE(s,\hat 
z(s)),\partial\calR(0))$. Then from \eqref{est.lambda-convex-energy} for every 
$s\in [0,S)$ and $h>0$ (such that $s+h\in [0,S]$ and with $\triangle_h\hat 
z(s)=\hat z(s+h) - \hat z(s)$) we obtain 
\begin{multline}
\label{est.lower-engest}
 \calI(\hat z(s+h)) - \calI(\hat z(s)) 
 \\ \geq 
 \langle \rmD\hat\calE(s,\hat z(s)) + \mu(s),
 \triangle_h \hat z(s)
 \rangle 
 +\langle\hat \ell(s),
 \triangle_h \hat z(s)
 \rangle 
 - \langle \mu(s),
 \triangle_h\hat z(s)
 \rangle 
\\
-C\norm{\triangle_h\hat z(s)}_\calV\calR(\triangle_h\hat z(s))\,.
\end{multline}
Thanks to the definition of $\mu$ we have the estimates 
\begin{multline*}
 \langle \rmD\hat\calE(s,\hat z(s)) + \mu(s),
 \triangle_h \hat z(s)
 \rangle 
 \\
 \leq \norm{\rmD\hat\calE(s,\hat z(s)) + 
\mu(s)}_{\bbV^*}\norm{\triangle_h\hat z(s)}_\bbV
= \dist_{\bbV}(-\rmD\hat\calE(s,\hat 
z(s)),\partial\calR(0))\norm{\triangle_h\hat z(s)}_\bbV
\end{multline*}
and $\calR(v)\geq \langle \mu(s),v\rangle$ 
for all $v\in \calZ$. Hence, after rearranging the terms in 
\eqref{est.lower-engest} and 
integration with respect to $s$, for $\sigma_1<\sigma_2\leq S-h$ we find 
\begin{multline*}
 \int_{\sigma_1}^{\sigma_2} h^{-1}(\calI(\hat z(s+h)) - \calI(\hat z(s)))\ds 
 \\
 +
 \int_{\sigma_1}^{\sigma_2} (1 + C\norm{\triangle_h\hat 
z(s)}_\calV)\calR(h^{-1}\triangle_h\hat z(s)) 
+ \dist_{\bbV}(-\rmD\hat\calE(s,\hat 
z(s)),\partial\calR(0))\norm{h^{-1}\triangle_h\hat z(s)}_\calV\ds 
\\
\geq \int_{\sigma_1}^{\sigma_2}\langle \hat \ell(s), 
h^{-1}\triangle_h \hat z(s)\rangle\ds\,.
\end{multline*}
The next aim is to pass to the limit $h\searrow 0$ in this energy 
dissipation estimate.
 Lemma \ref{lem:kurz-convdiffquot} implies that $\lim_{h\searrow 0} 
\int_{\sigma_1}^{\sigma_2}\langle \hat \ell(s), h^{-1}\triangle_h 
\hat z(s)\rangle\ds 
=\int_{\sigma_1}^{\sigma_2} \langle \hat \ell(s),\rmd \hat 
z(s)\rangle$. Moreover, since 
$s\mapsto\calI(\hat z(s))$ is uniformly continuous (cf.\ Proposition 
\ref{prop.param-conv}), 
for the first term on the left hand side we obtain 
\[
\lim_{h\searrow 0} \int_{\sigma_1}^{\sigma_2}h^{-1}(\calI(\hat 
z(s+h)) - \calI(\hat 
z(s)))\ds =\calI(\hat z(\sigma_2))-\calI(\hat z(\sigma_1)).
\]
Since $\hat 
z\in C([0,S];\calV)$, we obtain  $\triangle_h\hat z(s)\to 0$ strongly in 
$\calV$ and uniformly in $s$. Furthermore, since $z\in 
AC^\infty([0,S];\calX)$,   the limit 
$\lim_{h\searrow 0} \calR(h^{-1}\triangle_h z_h)$ exists for almost all $s$ and 
equals to $\calR[\hat z'](s)$, cf.\ Appendix \ref{app.bvac}. By the Lebesgue 
Theorem we thus obtain 
\[
 \lim_{h\searrow 0}
 \int_{\sigma_1}^{\sigma_2}
 (1 + C\norm{\triangle_h\hat 
z(s)}_\calV)\calR(h^{-1}\triangle_h\hat z(s)) \ds 
= \int_{\sigma_1}^{\sigma_2}\calR[\hat z'](s)\ds. 
\]
The definition of $G$ and that fact that $\vve(\hat\ell(s),\hat z(s))$ and 
$\vve(\hat\ell(s\pm),\hat z(s))$ differ in at most countably many points 
imply that  $\vve(\hat\ell(s),\hat z(s))=0$ for almost all $s\in 
[0,S]\backslash G$. 
Thus,
\begin{align*}
 \int_{\sigma_1}^{\sigma_2}\!\!\!\!\!
 \dist_{\bbV}(-\rmD\hat\calE(s,\hat 
z(s)),\partial\calR(0))
\bnorm{\tfrac1h\triangle_h\hat z(s)}_\calV\ds 
=\int_{(\sigma_1,\sigma_2)\cap G}\!\!\!\!\!\!\!\!\!\!\!\!
\vve(\hat\ell(s),\hat z(s))
\norm{\tfrac1h\triangle_h\hat z(s)}_\calV\ds .
\end{align*}
Since $\hat z\in W^{1,\infty}_\text{loc}(G;\calV)$, by Lebesgue's theorem we 
deduce for each $K\Subset G$ 
\begin{align*}
 \lim_{h\searrow 0} \int_{(\sigma_1,\sigma_2)\cap K} 
 \vve(\hat\ell(s),\hat z(s))
\norm{h^{-1}\triangle_h\hat z(s)}_\calV\ds 
= \int_{(\sigma_1,\sigma_2)\cap K} 
\vve(\hat\ell(s),\hat z(s))
\norm{\hat z'(s)}_\calV\ds. 
\end{align*}
To summarize, we have shown the following: By continuity of 
$\calI(\hat z(\cdot))$ and taking into account Proposition 
\ref{prop:kurz-conv}, 
for all $(a,b)\subset G$ we have 
\begin{align}
 \calI(\hat z(b)) -\calI(\hat z(a)) 
 + \int_a^b \calR[\hat z'(s)] + \vve(\hat \ell(s),\hat z(s))\norm{\hat 
z'(s)}_\bbV\ds 
\geq \int_a^b \langle \hat\ell(s),\rmd\hat z(s)\rangle,
\end{align}
while for every $[\alpha,\beta]\subset[0,S]\backslash G$ 
\begin{align}
 \calI(\hat z(\beta)) -\calI(\hat z(\alpha)) 
 + \int_\alpha^\beta \calR[\hat z'(s)] \ds
 \geq \int_\alpha^\beta \langle \hat\ell(s),\rmd\hat 
z(s)\rangle.
\end{align}
Since $G$ is the disjoint union of at most countably many (relatively) 
open intervals and keeping in mind 
\cite[Proposition 1.4]{KrejciLiero09}, a telescopic sum argument finally 
implies  that for all 
$\sigma_1 <\sigma_2\in [0,S]$ the energy dissipation estimate 
\begin{multline*}
 \calI(\hat z(\sigma_2)) -\calI(\hat z(\sigma_1)) 
 + \int_{\sigma_1}^{\sigma_2} \calR[\hat z'(s)] \ds 
 +
 \int_{(\sigma_1,\sigma_2)\cap G}\vve(\hat \ell(s),\hat 
z(s))\norm{\hat 
z'(s)}_\bbV\ds 
\\
\geq \int_{\sigma_1}^{\sigma_2} \langle \hat\ell(s),\rmd\hat z(s)\rangle
\end{multline*}
is valid.
Together with the opposite estimate (i.e.\ 
\eqref{eq.energydissipidentitylimitpparam} with $\leq$) we finally obtain 
\eqref{eq.energydissipidentitylimitpparam} with an equality. 

Improved convergences: By standard arguments it follows that in fact for all 
$s\in [0,S]$ it holds
\begin{gather*} 
\lim_{\varepsilon\to 0} \calI(\hat z_\varepsilon(s))=\calI(\hat z(s)),\\
\lim_{\varepsilon\to 0}\int_0^s\calR(\hat z_\varepsilon'(r))\dr 
=\int_0^s\calR[\hat z'](r)\dr,\\
\lim_{\varepsilon\to 0} \int_0^s \norm{\hat z_\varepsilon'(r)}_\bbV
\dist_\bbV(-\rmD\calE(\hat t_\varepsilon(r),\hat 
z_\varepsilon(r)),\partial\calR(0))\dr \\=\int_{(0,s)\cap G} 
\norm{\hat z'(r)}_\bbV
\dist_\bbV(-\rmD\hat\calE(r,\hat 
z(r)),\partial\calR(0))\dr\,.
\end{gather*}
In order to prove that the limit solution is normalized, i.e.\ in order to 
verify \eqref{eq.normalized.sol}, we rewrite $\int_0^s \calR(\hat 
z_\varepsilon'(r)) + \norm{\hat 
z_\varepsilon'(r)}_\bbV\dist_\bbV(-\rmD\calE(\hat t_\varepsilon(r),\hat 
z_\varepsilon(r)),\partial\calR(0))\dr =\int_0^s (1-\hat t_\varepsilon'(r))\dr$ 
and use the above convergences to conclude.
\end{proof}
\begin{definition}
\label{def.pparamsol-version2}
 Assume \eqref{eq.assumptionvv1}. A tuple $(S,\hat t,\hat z,\hat \ell)$ with 
 $S>0$, $\hat t\in W^{1,\infty}((0,S);\R)$, $\hat z\in 
AC^\infty([0,S];\calX)\cap L^\infty((0,S);\calZ)$ and 
  $\hat\ell\in BV([0,S];\calV^*)$ 
  is a normalized, $\vvp$-parameterized balanced 
viscosity solution  of the 
rate-independent system associated with $(\calI,\calR,\ell,z_0)$ if 
$\hat \ell$ is of the form \eqref{eq.charhatell}, if there exists an open set 
$G\subset[0,S]$ such that $\hat z\in W^{1,1}_{\text{loc}}(G;\calV)$, 
   $\rmD\hat\calE(\cdot,\hat z(\cdot))\in L^\infty_\text{loc}(G;\calV^*)$  
and such 
that    
$\vvm(\hat\ell(s),\hat z(s))>0$ for all $s\in G$  and $\vvm(\hat\ell(s),\hat 
z(s))=0$ for all  $s\in [0,S]\backslash G$, and if
\eqref{eq.pparamsol-con1}--\eqref{eq.energydissipidentitylimitpparam} are 
satisfied.  

With $\calL(\ell,z_0)$ we denote the set of normalized, $\vvp$-parameterized 
balanced viscosity solutions  associated with $(\calI,\calR,\ell,z_0)$.
\end{definition}
If \eqref{eq.assumptionvv1} is satisfied then by Theorem \ref{ex.ri.pparam} the 
set   $\calL(\ell,z_0)$ is not empty. 
\section{Properties of the solution set} 
\label{s-properties}
The next lemma shows that all elements of $\calL(\ell,z_0)$ enjoy the same 
regularity properties as the limit functions obtained in Proposition 
\ref{prop.param-conv} (except possibly the $BV([0,S];\calZ)$ regularity) with 
bounds that are uniform with respect to the set 
$\calL(\ell,z_0)$. While estimates 
\eqref{eq.unif-est1-paramsola}--\eqref{eq.unif-est1-paramsolb} here below are 
immediate  
consequences of the energy dissipation balance 
\eqref{eq.energydissipidentitylimitpparam} and the 
normalization property \eqref{eq.normalized.sol}, the uniform $L^\infty$-bound  
for $\rmD\hat \calE$, i.e.\ \eqref{eq.unif-est1-paramsolc},  requires a more 
refined analysis.  

\begin{lemma}
 \label{lem.reg.paramsolv1}
 
Assume \eqref{eq.assumptionvv1}. 

Every normalized, $\vvp$-parameterized 
balanced 
viscosity solution $(S,\hat t,\hat z,\hat \ell)\in \calL(\ell,z_0)$  of the 
rate-independent 
system 
associated with 
$(\calI,\calR,\ell,z_0)$ (according to Definition~\ref{def.pparamsol-version2}) 
satisfies 
\begin{enumerate}
\item  $ \calI(\hat z(\cdot))$ belongs to $C([0,S];\R)$.
\item 
 $\hat t$ is constant on the closure of each connected component of $G$ and 
there 
exists a measurable function $\lambda:(0,S)\to[0,\infty)$ with 
$\lambda(s)=0$ on $(0,S)\backslash G$ such that on each connected 
component $(a,b)\subset G$ the differential inclusion
\begin{align} 
\label{eq.diffincllambda}
 0\in\partial\calR(\hat z'(s)) +\lambda(s)\bbV\hat z'(s) + 
\rmD\hat\calE(s,\hat z(s))
\end{align}
is satisfied, for almost all $s\in (a,b)$. 
\\
For almost all $s\in G$ we have 
$\lambda(s)=\dist_\bbV(-\rmD\hat\calE(s,\hat z(s)),\partial\calR(0))/\norm{\hat 
z'(s)}_\bbV$. 
\item Estimates: \\ There exists a constant $c>0$ (depending on 
$\norm{z_0}_\calZ$, $\dist_\bbV(-\rmD\calE(0,z_0),\partial\calR(0))$,
 $\norm{\ell}_{L^\infty(0,T;\calV^*)}$, $\Var_{\calV^*}(\ell,[0,T])$, and 
$\diam_{\calV^*}(\partial\calR(0))$, only) such that 
for all normalized, $\vvp$-parameterized balanced viscosity solutions 
associated 
with $(z_0,\ell)$ it 
holds  $\rmD\hat \calE(\cdot,\hat z(\cdot))\in L^\infty((0,S);\calV^*)$ and 
\begin{gather}
 \label{eq.unif-est1-paramsola}
\norm{\hat z}_{L^\infty((0,S);\calZ)}\leq c, \,\, S\leq c,\\
\int_0^S 
\calR[\hat z'](s)\ds + \int_{(0,S)\cap G}\norm{\hat 
z'(s)}_\bbV\dist_\bbV(-\rmD\hat \calE(s,\hat z(s)),\partial\calR(0))\ds\leq c.
\label{eq.unif-est1-paramsolb}
\\
\bnorm{\rmD\hat \calE(\cdot,\hat z(\cdot))}_{L^\infty((0,S);\calV^*)}\leq c,
\label{eq.unif-est1-paramsolc}
\\
\norm{\lambda \bbV\hat z'}_{L^\infty(G;\calV^*)}\leq c.
\label{eq.unif-est1-paramsold}
\end{gather}
Finally, $\rmD\calI(\hat 
z(\cdot))\in C_\text{weak}([0,S];\calV^*)$. 
\end{enumerate} 
\end{lemma}

\begin{proof}
Continuity of $\calI(\hat z(\cdot))$ (claim (1)): 
The energy dissipation identity \eqref{eq.energydissipidentitylimitpparam} 
and the normalization property \eqref{eq.normalized.sol} imply that for all 
$a, b\in (0,S)$ we have 
\begin{align*}
 \abs{\calI(\hat z(b)) -\calI(\hat z(a)) }\leq \abs{b-a} + \abs{\hat t(b)-\hat 
t(a)} + 
\abs{\int_a^b\langle\hat\ell(s),\rmd\hat z(s)\rangle}.
\end{align*}
Since $\hat z\in C([0,S];\calV)$ (cf.\ Proposition 
\ref{prop.hellyarzasc-version2}) and taking into account 
estimate \eqref{est:kurz01},  the latter integral can be estimated as 
\begin{multline*}
\abs{\int_a^b \langle\hat\ell(s),\rmd\hat z(s)\rangle} 
=\abs{ \int_a^b 
\langle\hat\ell(s),\rmd(\hat z(s)-\hat z(a))\rangle
}
\\
\leq (\bnorm{\hat\ell}_{L^\infty((0,S);\calV^*} + \Var_{\calV^*}(\hat 
\ell,[0,S]))\norm{\hat z(\cdot) - \hat z(a)}_{C([a,b];\calV)}=:f(b).
\end{multline*}
Since $\lim_{b\to a}f(b)=0$, the  continuity of the 
mapping 
$s\mapsto\calI(\hat z(s))$ ensues.
%
%

Proof of claim (2): Since $m(\hat\ell(s),\hat 
z(s))>0$  on $G$, from the complementarity 
condition \eqref{eq.complementarity} we deduce that $\hat t$ is constant on 
each connected component of $G$. 
In order to verify \eqref{eq.diffincllambda}, let $[a,b]\Subset G$. Since by 
assumption $\hat z\in W^{1,1}((a,b);\calV)$, the identities $\calR[\hat 
z'](s)=\calR(\hat z'(s))$ and $\int_\alpha^\beta\langle \hat \ell(s),\rmd\hat 
z(s)\rangle =\int_\alpha^\beta \langle\hat\ell(s),\hat z'(s)\rangle\ds$ are 
valid for almost all $s\in (a,b)$ and all $\alpha<\beta\in (a,b)$, cf.\  
 \cite[Proposition 1.10]{KrejciLiero09}. 
Thus, localizing the energy dissipation identity 
\eqref{eq.energydissipidentitylimitpparam} 
(we apply the integrated version of the  chain rule 
\eqref{eq.int.chain.rule} and exploit the continuity of $\calI(\hat z(\cdot))$ 
provided in the first part of the proposition) yields
\begin{align}
\label{eq.locengdissipid}
\calR(\hat z'(s)) + \langle\rmD\calI(\hat z(s))-\hat\ell(s),\hat 
z'(s)\rangle_{\calV^*,\calV} + \norm{\hat z'(s)}_\bbV\dist_\bbV(-\rmD\hat 
\calE(s,\hat z(s)),\partial\calR(0)) =0
\end{align}
which is valid for almost all $s\in (a,b)$. Since $\hat t$ is constant 
on $(a,b)$, from  \eqref{eq.normalized.sol} it follows that $\hat z'(s)\neq 0$ 
almost everywhere on $(a,b)$. Hence, with 
\[
\lambda(s):= 
\begin{cases}
\dist_\bbV(-\rmD\hat\calE(s,\hat z(s)),\partial\calR(0))/\norm{\hat 
z'(s)}_\bbV&\text{if }\hat z'(s)\neq 0\\
0&\text{otherwise}
\end{cases}
\]
we have $ \norm{\hat z'(s)}_\bbV\dist_\bbV(-\rmD\hat 
\calE(s,\hat z(s)),\partial\calR(0)) =\langle \lambda(s)\bbV\hat z'(s),\hat 
z'(s)\rangle$ and \eqref{eq.diffincllambda} follows from 
\eqref{eq.locengdissipid} and the one-homogeneity of $\calR$. This finishes the 
proof of claim (2) in Lemma~\ref{lem.reg.paramsolv1}.

Proof of the estimates (claim (3)): 
The verification of 
\eqref{eq.unif-est1-paramsola}--\eqref{eq.unif-est1-paramsolb} starts from the  
energy 
dissipation identity
 \eqref{eq.energydissipidentitylimitpparam}.  
Indeed, for all $b\in [0,S]$ we deduce relying on 
the coercivity estimate for $\calI$ and on \cite[Theorem 1.9]{KrejciLiero09} 
(cf.\ \eqref{est:kurz01} in the Appendix)
\begin{align*}
 \tfrac{\alpha}{2}\norm{\hat z(b)}^2_\calZ\leq \calI( z_0) + 
(\norm{\ell}_{L^\infty(0,T;\calV^*)}+\Var_{\bbV^*}(\ell,[0,T])\norm{\hat 
z}_{L^\infty(0,S;\calZ)}.
\end{align*}
Here, we also used that 
$\Var_{\bbV^*}(\ell,[0,T])=\Var_{\bbV^*}(\hat\ell,[0,S])$. 
From this the claimed uniform bounds in 
\eqref{eq.unif-est1-paramsola}--\eqref{eq.unif-est1-paramsolb} are an 
immediate consequence taking into account the normalization condition 
\eqref{eq.normalized.sol}.

Let us finally show the higher regularity of $\rmD\hat\calE(\cdot,\hat 
z(\cdot))$ along with estimate \eqref{eq.unif-est1-paramsolc}. 
Observe that  $m(\hat \ell(s),\hat z(s))=0$ for  all $s\in 
(0,S)\backslash 
G$. Since $\hat\ell(s),\hat \ell(s+),\hat\ell(s-)$ differ in at most countably 
many points, this implies that $\dist_\bbV(-\rmD\hat \calE(s,\hat 
z(s)),\partial\calR(0))=0$ almost everywhere on $(0,S)\backslash G$. Since 
$\partial\calR(0)$ is a bounded subset of $\calV^*$, for almost all $s\in 
[0,S]\backslash G$  we 
 obtain $\bnorm{\rmD\hat\calE(s,\hat z(s))}_{\calV^*}\leq 
\diam_{\calV^*}(\partial\calR(0))$, which is 
\eqref{eq.unif-est1-paramsolc} restricted to the set $(0,S)\backslash G$. 

The regularity and the estimate with respect to the set  $G$ will be 
deduced by a rescaling argument relying on the differential inclusion 
 \eqref{eq.diffincllambda}, Proposition \ref{prop.unifboundeps} and Remark 
\ref{rem.viscindepT}. 
Let $(a,b)\subset G$ be a nonempty maximal connected component of $G$. 
A proof by contradiction relying on the lower semi-continuity property of 
$\vvm (\cdot,\cdot)$ stated in Lemma \ref{lem:vvmlsc} shows that for every 
compact $K\Subset(a,b)$ there exists $c_K>0$ such that $\vvm(\hat\ell(s),\hat 
z(s))\geq c_K$ for all $s\in K$. From the normalization condition we thus 
obtain $\norm{\hat z'(s)}_\calV\leq c_K^{-1}$ almost everywhere on $K$ and 
hence $\lambda(s)\geq c_K^2>0$ on $K$.  Thus $\lambda^{-1}\in 
L^\infty_\text{loc}(a,b)$. 

We now distinguish two cases, namely case (a), where there exists $s_*\in 
(a,b)$ such that $\lambda^{-1}\notin L^1((a,s_*))$ and the simpler case (b), 
where we assume that for all $s_*\in (a,b)$ the function $\lambda^{-1}$ 
belongs to $L^1((a,s_*))$. 

Case (a): Assume that $\lambda^{-1}\notin L^1((a,s_*))$. 
Since $\lambda^{-1}\in 
L^\infty_\text{loc}(a,b)$,  
for every $\varepsilon>0$ there exists 
$c_\varepsilon>0$ such that  
$\lambda^{-1}\big|_{(a+\varepsilon,s_*)}\leq  c_\varepsilon$. Since 
by assumption $\lambda^{-1}$ is not integrable on $(a,s_*)$, $\lambda^{-1}$ 
is   unbounded towards the point $a$. To be more precise, for every $n\in \N$ 
the set 
\begin{align*}
S_n:= \Set{s\in (a,a +\tfrac{1}{n})}{\tfrac{1}{\lambda(s)}\geq n}
=\Set{s\in (a,a +\tfrac{1}{n})}{\lambda(s)\leq \tfrac{1}{n}}
\end{align*}
has positive Lebesgue measure. 
Moreover, taking into account the normalization property 
\eqref{eq.normalized.sol} and the structure of $\lambda$, we deduce
\begin{align}
\label{eq.prooflinfbound1}
 \text{for all $n\in \N$ and almost all $s\in S_n$}\quad 
 \dist_\bbV(- \rmD\hat\calE(s,\hat z(s)),\partial\calR(0))\leq 
\tfrac{1}{\sqrt{n}}\,.
\end{align}
Let now $s_n\in S_n$ such that $\dist_\bbV(- \rmD\hat\calE(s_n,\hat 
z(s_n)),\partial\calR(0))\leq \tfrac{1}{\sqrt{n}}$. Without loss of generality  
we 
assume that the sequence $(s_n)_{n\in\N}$ is  decreasing and 
converging to $a$.  Observe that  $- \rmD\hat\calE(s_n,\hat 
z(s_n))\notin \partial\calR(0)$ 
 for all $n$ since $\vvm(\hat\ell(s_n),\hat 
 z(s_n))>0$ on $G$. 
Observe further that  $\hat z$ satisfies 
the following initial value problem with $z_{0,n}:=\hat z(s_n)$
\begin{gather*}
 0\in \partial\calR(\hat z'(s)) + \lambda(s)\bbV\hat z'(s) 
 +\rmD\hat \calE(s,\hat 
z(s)),\quad s\in (s_n,b),\\
\hat z(s_n)=z_{0,n},\quad \rmD\hat\calE(s_n,z_{0,n})\in 
\calV^*.
\end{gather*}
We next rescale this system as follows: For $s\in [s_n,b)$ let 
$\Lambda(s):=\int_{s_n}^s \tfrac{1}{\lambda(r)}\dr$. The above considerations 
show 
that $\Lambda$ is well defined for all $s\in [s_n,b)$. However, 
for $s\nearrow b$ one might have $\Lambda(s)\to \infty$.  Moreover, $\Lambda$ 
is strictly increasing, 
continuous and the inverse function 
$\sigma:=\Lambda^{-1}:[0,\Lambda(b))\to[s_n,b)$ exists. For $r\in 
[0,\Lambda(b))$ 
let $\wt z(r):=\hat z(\sigma(r))$, $\wt\ell(r)=\hat \ell(\sigma(r))$ 
and 
$\wt\calE(r,v)=\hat \calE(\sigma(r),v)=\calI(v) - 
\langle\wt\ell(r),v\rangle$. The function $\wt z$ solves the 
Cauchy problem
\begin{gather*}
 0\in \partial\calR(\wt z'(r)) + \bbV\wt z'(r) +\rmD\wt\calE(r,\wt  
z(r)),\quad r\in (0,\Lambda(b)),\\
\wt z(0)=z_{0,n},\quad \rmD\wt\calE(0,\wt z(0))\in \calV^*.  
\end{gather*}
Thus, Proposition \ref{prop.unifboundeps} and Remark \ref{rem.viscindepT} are  
applicable and imply in 
particular that $\rmD\calI(\wt z)\in L^\infty((0,\Lambda(b));\calV^*)$ with a 
bound that depends on $\norm{\wt z(0)}_\calZ$, $\Var_{\calV^*}(\wt 
\ell;[0,\Lambda(b)])$, 
$\bnorm{\wt\ell}_{L^\infty(0,\Lambda(b);\calV^*)}$ and 
$\dist_\bbV(-\rmD\wt\calE(0,\wt z(0)),\partial\calR(0))$,  
only. This immediately translates into $\rmD\calI(\hat 
z)\in L^\infty((s_n,b);\calV^*)$ with 
\begin{align*}
& \norm{\rmD\calI(\hat z)}_{L^\infty((s_n,b);\calV^*)}
\\
&\leq c\Big(\norm{\hat 
z(s_n)}_{\calZ} + \Var_{\calV^*}(\hat \ell,[s_n,b]) 
+\bnorm{\hat\ell}_{L^\infty(s_n,b;\calV^*)}
+\dist_\bbV(-\rmD\hat\calE(s_n,\hat z(s_n)),\partial\calR(0))
\Big)\\
&\leq c\big(\norm{\hat z}_{L^\infty((0,S);\calZ)} +\Var_{\calV^*}(\ell;[0,T]) +
\norm{\ell}_{L^\infty(0,T;\calV^*)} +  
\frac{1}{\sqrt{n}}\big),
\end{align*}
and the constant $c$ is independent of the chosen solution $\hat z$ and 
of $s_n$. For $n\to\infty$ we ultimately obtain $\rmD\calI(\hat z)\in 
L^\infty((a,b);\calV^*)$ with a bound that depends on the data $z_0,\ell$, only.

Case (b): Now we assume that $\lambda^{-1}\in L^1((a,s_*))$ for every $s_*\in 
(a,b)$. Since $G$ is open and since (by assumption) $(a,b)$ is a maximal 
connected component of 
$G$, we have $a\notin G$ and hence, $\vvm(\hat\ell(a),\hat z(a))=0$. 
As above, we  rescale the 
equation by applying  the following transformation:  Let 
$\Lambda(s):=\int_a^s\frac{1}{\lambda(r)}\dr$ and 
 $\sigma:=\Lambda^{-1}$ its inverse function. For $r\in 
(0,\Lambda(b))$ we define $\wt z(r):=\hat z(\sigma(r))$ and 
$\wt\ell(r):=\hat\ell(\sigma(r))$. 
The function $\wt z$ satisfies the 
initial value problem 
\begin{align*}
\wt z(0)=\hat z(a),\quad 0\in\partial\calR(\wt z'(r)) + \bbV\wt z'(r) + 
\rmD\calI(\wt  
z(r)) -\wt\ell(r) \text{ for a.a.\ } r\in (0,\Lambda(b))
\end{align*}
with $\rmD\calI(\wt z(0))-\wt\ell(0)\in \calV^*$.  
By Proposition \ref{prop.unifboundeps} we have   
 $\rmD\calI(\wt z)\in L^\infty((0,\Lambda(b));\calV^*)$ with a 
bound depending only on $\norm{\hat z(a)}_\calZ$,  on 
$\dist_\bbV(-\rmD\hat\calE(a,\hat z(a)),\partial\calR(0))$ 
and on $\Var_{\calV^*}(\wt 
\ell,[0,\Lambda(b)])$. This immediately carries over to  
$\rmD\calI(\hat z)\in L^\infty(a,b;\calV^*)$ with the same bound. 
Observe that there exists $\ell_*\in \{\hat\ell(a),\hat \ell(a+),\hat 
\ell(a-)\}$ with $-\rmD\calI(\hat z(a))+\ell_*\in \partial\calR(0)$. Hence, 
\begin{multline*}
\dist_\bbV(-\rmD\hat\calE(a,\hat z(a)),\partial\calR(0)) 
\\
\leq \bnorm{-\rmD\hat\calE(a,\hat z(a)) + (\rmD\calI(\hat z(a)) - 
\ell_*)}_{\bbV^*}
\leq c_\bbV\big(
\norm{\ell}_{L^\infty(0,T;\calV^*)} + \Var_{\calV^*}(\ell,[0,T])
\big).
\end{multline*}

Combining the estimates derived for the cases (a) and (b) with the estimate 
derived for $(0,S)\backslash G$ we 
ultimately arrive at \eqref{eq.unif-est1-paramsolc}. 
Now, \eqref{eq.unif-est1-paramsold} is an immediate consequence of 
\eqref{eq.diffincllambda} and the estimate \eqref{eq.unif-est1-paramsolc}.

Finally, thanks to Proposition \ref{prop.hellyarzasc-version2}, $\hat z\in 
C_\text{weak}([0,S];\calZ)$, and hence, $\rmD\calI(\hat z(\cdot)) 
\in C_\text{weak}([0,S];\calZ^*)$ (by assumption 
\eqref{ass.fweakconv}). Together with the uniform bound of $\rmD\calI(\hat 
z(\cdot))$ in $\calV^*$ the last assertion of claim (3) follows. 
\end{proof}

\begin{remark}
Let $(S,\hat t,\hat z,\hat \ell)$ be a solution associated with 
$(\calI,\calR,\ell,z_0)$ in the sense of Definition 
\ref{def.pparamsol-version2}. Let $\hat \ell_\pm$ be the left resp.\ the right 
continuous version of $\hat \ell$. Then $(S,\hat t,\hat\ell_\pm,\hat z)$ is 
 a solution associated with $(\calI,\calR,\ell,z_0)$  in the sense of 
Definition \ref{def.pparamsol-version2}, as well. 

This can be seen as follows:  $\hat \ell$ and its left or right 
continuous version differ in at most countably many points. 
Thus, \eqref{eq.pparamsol-con1}--\eqref{eq.complementarity} are valid after 
replacing $\hat \ell$ with $\hat\ell_\pm$.  
Let $G_\pm:=\Set{s\in [0,S]}{\vvm(\hat \ell_\pm(s),\hat z(s))>0}$. Clearly, 
$G\subseteq G_\pm$ and the sets differ on a set of measure zero, only. 
Since 
$\hat z\in C([0,S];\calV)$ (cf.\ Proposition \ref{prop.hellyarzasc-version2}), 
for every $s\in [0,S]$ we have $\int_0^s\langle\hat\ell(r),\rmd\hat 
z(r)\rangle_{\calV^*,\calV} = \int_0^s\langle\hat\ell_\pm(r),\rmd\hat 
z(r)\rangle_{\calV^*,\calV}$. This is due to the identity 
$\int_a^b\langle\chi_{s_*}(r),\rmd g(r)\rangle_{\calV^*,\calV} = g(s_* 
+)-g(s_*-)$ that is valid for every $s_*\in [a,b]$ and every regulated function 
$g\in G([a,b];\calV)$, \cite[Proposition 2.1]{tvrdy89}. Here, $\chi_{s_*}(s)= 
0$ if $s\neq s_*$ and  $\chi_{s_*}(s_*)=1$. Hence, the energy dissipation 
identity \eqref{eq.energydissipidentitylimitpparam} remains unaffected by a 
switch from $\hat \ell$ to $\hat\ell_\pm$. 

As a consequence of the weak continuity of $\rmD\calI(\hat z(\cdot))$ 
in $\calV^*$ (see Lemma \ref{lem.reg.paramsolv1}) with the same arguments as in 
the proof of Theorem \ref{ex.ri.pparam} it follows that $G_\pm$ is 
open. Thus, $\rmD\calI(\hat z(\cdot))\in L^\infty_\text{loc}(G_\pm;\calV^*)$. 
Moreover, condition   
\eqref{eq.normalized.sol} holds  with $G_\pm$ instead of $G$. 
It remains to show that $\hat z\in W^{1,1}_\text{loc}(G_\pm;\calV)$.  
Let $K\Subset G_\pm $ be compact. Then, again by lower semicontinuity, 
$\inf_{s\in K}\vvm(\hat\ell_\pm(s),\hat z(s))=:c>0$ which in turn implies 
(using the normalization property \eqref{eq.normalized.sol}) that $\norm{\hat 
z'(s)}_{\calV}\leq c$ a.e.\ on $K$. Since $z\in W^{1,1}_\text{loc}(G;\calV)$  
this implies  $\hat z\in W^{1,\infty}(K\cap G;\calV)$ and thus ultimately $\hat 
z\in W^{1,1}_\text{loc}(G_\pm;\calV)$. 
\end{remark}

\begin{proposition}
 \label{prop.properties_sol_set} 
Assume \eqref{eq.assumptionvv1}. The set $\calL(\ell,z_0)$ is compact in the 
following sense: For every sequence $(S_n,\hat t_n,\hat z_n,\hat\ell_n)_{n\in 
\N} \subset \calL(\ell,z_0)$ there exists a (not relabeled) subsequence and a 
tuple $(S,\hat t,\hat z,\hat \ell)\in \calL(\ell,z_0)$ such that 
\begin{align}
S_n\to S,\quad \hat t_n&\overset{*}{\rightharpoonup}\hat t 
\text{ weakly$*$ in }W^{1,\infty}(0,S),\,\, \hat t(S)=T,
\label{est.convparam6}
\\
\label{est.convparam7}
\hat z_n& \rightharpoonup\hat z \text{ weakly$*$ in 
}L^\infty(0,S;\calZ) \text{ and uniformly in }C([0,S];\calV), 
\\
 \hat\ell_n&\overset{*}{\rightharpoonup}\hat\ell,\,\, 
 \quad \rmD\calI(\hat 
z_n)\overset{*}{\rightharpoonup}\rmD\calI(\hat z) 
\text{ weakly$*$ in }L^\infty(0,S;\calV^*),  
\label{est.convparam8}
\end{align}
and for every $s\in [0,S]$
\begin{align}
 \hat t_n(s)&\to\hat t(s), 
 \quad 
 \hat z_n(s)\rightharpoonup\hat z(s) \text{ weakly in $\calZ$ }, 
 \label{est.convparam9}\\
 \rmD\calI(\hat z_n(s))&\rightharpoonup\rmD\calI(\hat z(s)) \text{ 
weakly in $\calV^*$}, \quad 
\hat\ell_n(s)\rightharpoonup\hat\ell(s) \text{ weakly in }\calV^*. 
\label{est.convparam10}
\end{align} 
 
\end{proposition}
\begin{proof}
Let $(S_n,\hat t_n,\hat z_n,\hat\ell_n)_{n\in 
\N} \subset \calL(\ell,z_0)$ and let $(G_n)_n\subset[0,S]$ be the 
corresponding 
open sets according to Definition \ref{def.pparamsol-version2}.  
 Thanks to Lemma \ref{lem.reg.paramsolv1} the bounds 
\eqref{eq.unif-est1-paramsola}--\eqref{eq.unif-est1-paramsold} hold uniformly 
with respect to $n$ and $G_n$. Hence, up to a subsequence, $S_n\to S$ for 
some $S>0$. Again, if $S>S_n$ we  extend all functions by their constant value 
at $S_n$. Having in mind the normalization condition 
\eqref{eq.normalized.sol}, with Lemma \ref{prop.hellyarzasc-version2}, part 
(b), 
 there exists $\hat z\in AC^\infty([0,S];\calX)\cap 
L^\infty((0,S);\calZ)$, $\hat t\in W^{1,\infty}(0,S)$ and $\hat\ell\in 
BV([0,S];\calV^*)$  such that (up to extracting a further subsequence)  the 
convergences in \eqref{est.convparam6}--\eqref{est.convparam10} hold. Thereby, 
the convergences of the sequence $\hat \ell_n$ follows again from the Banach 
space valued version of Helly's selection principle \cite{BarbuPrecupanu86}, 
while the convergences of $\rmD\calI$ follow by the same arguments as in the 
proof of Proposition \ref{prop.param-conv}.  Moreover, again by the 
same arguments as in Proposition \ref{prop.param-conv} the continuity of   
$s\mapsto \calI(\hat z(s))$ ensues. Observe further that the 
function  $s\mapsto \rmD\calI(\hat z(s))$ belongs to 
$C_\text{weak}([0,S];\calV^*)$. 

The characterization of the limit function $\hat\ell$
follows by similar arguments as in the proof of Proposition 
\ref{prop.param-conv}. 
Indeed, since for the functions $\ell,\ell_-,\ell_+$ in each $t\in [0,T]$ the 
(strong) left and right limits exist and belong to $\{\ell_-(t),\ell_+(t)\}$ 
and since $\hat\ell_n(s)\in \{\ell(\hat t_n(s)), \ell_-(\hat 
t_n(s)),\ell_+(\hat t_n(s))\}$, the limit $\hat\ell(s)$ belongs to $\{\ell(\hat 
t(s)), \ell_-(\hat t(s)),\ell_+(\hat t(s))\}$ and we even have strong 
convergence $\hat \ell_n(s)\to\hat\ell(s)$ in $\calV^*$. 
If $t_*\in [0,T]$ is a point of continuity of  $\ell$, then from the above, for 
all $s\in \hat t^{-1}(t_*)$ we have 
$\hat\ell(s)=\ell(t_*)=\ell_-(t_*)=\ell_+(t_*)$. Assume now that $t_*$ is not a 
point of continuity of $\ell$ with $\ell_-(t_*)\neq \ell_+(t_*)$. Let $s\in 
[a,b]:=\hat t^{-1}(t_*)$ with $\hat\ell(s)=\ell_+(t_*)$.   A proof by 
contradiction shows that there exists $n_0\in \N$ such that for all $n\geq n_0$ 
we have $\hat t_n(s)\geq \hat t(s)=t_*$. Moreover, by monotonicity of the 
functions $\hat t_n$, for all $n\geq n_0$ and all $r\in[s,b]$ we have $\hat 
t_n(r)\geq t_*$. Hence, $\hat\ell_n(r)\to \ell_+(t_*)$, as well. Let 
$s_+:=\inf\Set{s\in [a,b]}{\hat\ell(s)=\ell_+(t_*)}$. In a similar way we 
define $s_-:=\sup \Set{s\in [a,b]}{\hat\ell(s)=\ell_-(t_*)}$ and obtain 
$\hat\ell_n(r)\to\ell_-(t_*)$ for all $r\in [a,s_-)$.  Thus we have shown that 
$\hat\ell(s)=\ell_-(s)$ if $s\in [a,s_-)$ and $\hat\ell(s)=\ell_+(s)$ if $s\in 
(s_+,b]$. Assume finally that $s_-<s_+$. Then $\hat\ell(s)=\ell(t_*)$ 
for all $s\in (s_-,s_+)$ and for each pair $s_1<s_2\in (s_-,s_+)$ there exists 
$n_0\in \N$ such that $\hat \ell_n(s_1)=\hat \ell_n(s_2)=\ell(t_*)$ for all 
$n\geq n_0$ (proof by contradiction). This implies in  particular that  $\hat 
t_n(s_1)=\hat t_n(s_2)=t_*$ for all $n\geq n_0$ and that $s_1=s_{*,n}$ and 
$s_2=s_{*,n}$ for all $n$ with $s_{*,n}$ from \eqref{eq.charhatell}. But this 
is 
a contradiction. Hence, $s_-=s_+$ in this case. For the case 
$\ell_-(t_*)=\ell_+(t_*)\neq \ell(t_*)$ the arguments can be easily adapted. 
To summarize, we finally have shown that 
$\hat \ell$ is of the structure \eqref{eq.charhatell}. 

It remains to prove that $(S,\hat t,\hat z,\hat\ell)\in \calL(\ell,z_0)$. Here, 
we follow mainly the proof of Theorem \ref{ex.ri.pparam}. 
Due to 
Proposition \ref{app_prop:lsc2} the complementarity relation 
\eqref{eq.complementarity} is satisfied by the limit tuple. 

Energy dissipation estimate \eqref{eq.energydissipidentitylimitpparam}, $\leq$: 
Starting from \eqref{eq.energydissipidentitylimitpparam} written for every $n$, 
by lower semicontinuity, the Helly convergence Theorem 
\cite[Theorem 3.2]{MaMi05}, Lemma~\ref{app.lembv1}, and Proposition 
\ref{prop:kurz-conv}, we obtain 
$\liminf_n\calI(\hat z_n(s))\geq \calI(\hat z(s))$, $\liminf_n 
\int_0^s\calR[\hat z_n'](r)\dr\geq \int_0^s\calR[\hat z'](r)\dr$ and 
$\int_0^s\langle\hat\ell_n(r),\rmd\hat z_n(r)\rangle \to 
\int_0^s\langle\hat\ell(r),\rmd\hat z(r)\rangle$.  
 
Let $G:=\Set{s\in [0,S]}{\vvm(\hat\ell(s),\hat z(s))>0}$. Like in the proof of 
Theorem \ref{ex.ri.pparam} it follows that $G$ is open with $0\notin G$. Let 
$K\subset G$ be compact. With the very same arguments as in the proof of 
Theorem \ref{ex.ri.pparam} there exists $n_0\in \N$ such that for all $n\geq 
n_0$ we have $K\subset G_n$ and $\sup_{n\geq n_0}\norm{\hat 
z'_n}_{L^\infty(K;\calV)}<\infty$.  
Hence, each subsequence of $(\hat z_n)_n$ 
contains a subsubsequence that converges  weakly$*$ in $W^{1,\infty}(K;\calV)$ 
to $\hat z$, whence $\hat z\in W^{1,\infty}_\text{loc}(G;\calV)$ and  in fact 
the 
whole sequence converges. By Proposition~\ref{app_prop:lsc} we therefore have 
the analogue to \eqref{est.proof-ex01}. 
In summary, we have proved 
\eqref{eq.energydissipidentitylimitpparam} with $\leq $ instead of equality. 
By similar arguments we obtain 
\eqref{eq.normalized.sol} with $\geq$ instead of equality. The very same 
arguments as in the proof of Theorem \ref{ex.ri.pparam} yield the opposite 
estimate in \eqref{eq.energydissipidentitylimitpparam} as well as the 
normalization condition \eqref{eq.normalized.sol}. Hence, in summary the limit 
tuple $(S,\hat t,\hat z,\hat \ell)$ belongs to the solution set 
$\calL(\ell,z_0)$.  
\end{proof}

\begin{appendix}

\section{Properties of $\calR$}
\label{app.R}
We collect here the properties of the dissipation $\calR:\calX\to[0,\infty)$ 
and related quantities which are used throughout the paper. 
Since $\calR$ is positively one-homogeneous functional, it holds 
\[
\eta\in \partial\calR(v)\quad\Leftrightarrow\quad
\begin{cases}
 \langle \eta,v\rangle =\calR(v)&\\
 \langle \eta, w\rangle\leq \calR(w)&\text{for all }w\in \calZ\,.
\end{cases}
\] 
It follows from \eqref{eq.Mief100} that
$\partial\calR(0)\subset\calV^*$ and bounded in $\calV^*$-norm (see for 
instance 
\cite{Knees2018}).

For $\varepsilon>0$, let $\calR_\varepsilon:\calV\to[0,+\infty)$, 
$\calR_\varepsilon(v):=\calR(v)+\frac{\varepsilon}{2}\langle \bbV v,v\rangle$ 
be 
the viscous regularized dissipation potential. 
Its Fenchel-Moreau conjugate with respect to $\calV-\calV^*$ duality, 
$\calR_\varepsilon^*:\calV^*\to[0,+\infty)$, is defined by 
$\calR_\varepsilon^*(\eta)=\sup\{\langle\eta,v\rangle_{\calV^*,\calV}
-\calR_\varepsilon(v):  
\ v\in\calV\}$ and can be explicitly described by
\[
\calR_\varepsilon^*(\eta)=
\frac{1}{2 \varepsilon}\big(\dist_\bbV(\eta,\partial\calR(0))\big)^2\,.
\]
By $\dist_\bbV(\cdot,\partial\calR(0))$ we denote the distance of an element of 
$\calV^*$ to $\partial\calR(0)(\subset\calV^*)$ measured in the norm 
induced by the operator $\bbV$: for $\eta\in\calV^*$,
\begin{equation}
\label{distV}
   \dist_\bbV(\eta,\partial\calR(0)) := \inf \{\norm{\eta - \xi}_{\bbV^{*}}: \ 
\xi\in\partial\calR(0)\}  \,,
\end{equation}
where $\norm{\sigma}^2_{\bbV^{*}}=\langle \sigma,\bbV^{-1}\sigma\rangle$.

\section{Kurzweil integrals and convergence}
\label{sec.Kurzweil}
In this section we use the terminology from \cite{KrejciLiero09}.  Let $\calW$ 
be a Banach space and let $G([a,b];\calW)$ denote the space of regulated 
functions $f:[a,b]\to \calW$, i.e.\ the space of those functions for which 
there exist both one-sided limits $f(t+), f(t-)\in \calW$ in every $t\in 
[a,b]$, see  \cite{Dieu69,KrejciLiero09}. 
For functions $f:[a,b]\to\calW^*$ and $g:[a,b]\to\calW$ we denote with 
$\int_a^b\langle f(t),\rmd g(t)\rangle$ ($\langle\cdot,\cdot\rangle$ the dual 
pairing of $\calW$) the Kurzweil integral of $f$ with respect to $g$.  
According to \cite[Theorem 1.9]{KrejciLiero09} (see also 
\cite[Section 2]{tvrdy89}), the 
Kurzweil integral of $f$ with respect to $g$ exists provided that $f\in 
G([a,b];\calW^*)$ and $g\in BV([a,b];\calW)$ or vice versa, i.e.\ $f\in 
BV([a,b];\calW^*)$ and $g\in G([a,b];\calW)$. In both cases the following 
estimate is valid 
\begin{multline}
  \label{est:kurz01}
  \abs{\int_a^b \langle f(t),\rmd g(t)\rangle}
   \leq \min\left\{ 
 \norm{f}_{L^\infty(a,b;\calW^*)}\Var_\calW(g,[a,b]), \right. 
 \\
\left.  (\norm{f(a)}_{\calW^*} + \norm{f(b)}_{\calW^*} + 
\Var_{\calW^*}(f,[a,b])) 
  \norm{g}_{L^\infty(a,b;\calW)}
\right\}\,.
\end{multline}

\begin{proposition}
\label{prop:kurz-conv}
For $n\in \N$ let $z, z_n\in C([a,b];\calW)$, $\ell,\ell_n\in 
BV([a,b];\calW^*)$ and assume that  $(z_n)_n$ converges uniformly to 
$z$. Assume further that 
\[
\sup_{n\in \N}\left(\norm{\ell_n}_{L^\infty((a,b),\calW^*)} + 
\Var_{\calW^*}(\ell_n,[a,b])\right)=:C<\infty
\] 
and 
that $\ell_n(t)\rightharpoonup \ell(t)$ weakly$*$ in $\calW^*$ for every $t\in 
[a,b]$. 
Then $\int_a^b\langle \ell_n(t),\rmd z_n(t)\rangle \to \int_a^b\langle \ell(t), 
\rmd z(t)\rangle$.  
\end{proposition}

\begin{proof}
Let $(\ell_n)_n,\ell$ be given according to Proposition \ref{prop:kurz-conv}. 
Observe  that by lower semicontinuity we obtain 
$\norm{\ell}_{L^\infty((a,b),\calW^*)} + 
\Var_{\calW^*}(\ell,[a,b])\leq C$. 
 Assume first that $z\in C^1([a,b];\calW)$.  By 
\cite[Prop.\ 1.10]{KrejciLiero09}, we have 
$\int_a^b\langle \ell_n(t),\rmd z(t)\rangle = (L)\int_a^b\langle \ell_n(t),\dot 
z(t)\rangle\dt$, where the right hand side denotes the Lebesgue integral. Due 
to the assumptions, the integrand converges pointwise for every $t$ and is 
uniformly bounded with respect to $t$ and $n$. Hence, by Lebesgue's Theorem 
 we have $(L)\int_a^b\langle \ell_n,\dot z\rangle\dt \to 
(L)\int_a^b\langle \ell,\dot 
z\rangle\dt=\int_a^b\langle \ell(t),\rmd z(t)\rangle$. 

Since $C^1([a,b];\calW)$ is dense in $C([a,b];\calW)$ with respect to the 
sup norm, this convergence carries over to the case 
$z\in C([a,b];\calW)$ in the usual way.  Indeed, let $z\in C([a,b];\calW)$ and 
choose $\varepsilon>0$ arbitrarily. Let $\wt z\in C^1([a,b];\calW)$ with 
$\norm{z-\wt z}_{L^\infty((a,b);\calW)}\leq \varepsilon/3$. Let 
$n_\varepsilon\in \N$ such that  we have 
$\abs{\int_a^b\langle \ell_n,\rmd \wt z\rangle - \int_a^b\langle \ell,\rmd \wt 
z\rangle}\leq \varepsilon C/3$ for all $n\geq n_\varepsilon$. By 
\eqref{est:kurz01}, for all $n\geq 
n_\varepsilon$ it follows 
\begin{multline*}
 \abs{\int_a^b\langle \ell_n,\rmd(\wt z -z)\rangle}
 \leq 
(\norm{\ell_n(a)}_{\calW^*} +\norm{\ell_n(b)}_{\calW^*} + 
\Var_{\calW^*}(\ell_n,[a,b]))\norm{\wt z-z}_{L^\infty(a,b;\calW)}
\leq \tfrac{C \varepsilon}{3},
\end{multline*}
and similar for $\ell$ instead of $\ell_n$. Thus, 
\begin{align*}
 \abs{\int_a^b \langle \ell_n - \ell,\rmd z\rangle}\leq 
 \abs{\int_a^b \langle \ell_n ,\rmd (z-\wt z)\rangle}
 +  \abs{\int_a^b \langle \ell_n -\ell ,\rmd \wt z\rangle}
 +  \abs{\int_a^b \langle \ell ,\rmd (z-\wt z)\rangle}
 \leq C\varepsilon\,,
\end{align*}
which proves the convergence with $z\in C([a,b],\calW)$ fixed. 
Finally, the very same argument provides the general statement of Proposition 
\ref{prop:kurz-conv}.
\end{proof}
\begin{lemma}
 \label{lem:kurz-convdiffquot}
 Let $g\in C([a,b];\calW)$ and $f\in BV([0,S];\calW^*)$. Then 
 \begin{align}
\label{eq.kurz-convdiffquot}
 \lim_{h\searrow 0} \int_a^{b-h}\langle f(s), 
h^{-1}(g(s+h)-g(s))\rangle
\ds
= \int_a^b\langle f(s),\rmd 
g(s)\rangle
\,.
 \end{align}
\end{lemma}
In the proof we use the notation $(\triangle_h g)(s):= g(s+h)-g(s)$ for 
$g:[a,b]\to\calW$ and $s,h\in \R$. Observe that the product rule 
$\triangle_h(fg)(s)=(\triangle_h f)(s) g(s+h) + f(s)(\triangle_h g)(s)$ is 
valid. 
\begin{proof}
 Let $f\in BV([a,b];\calW^*)$ and assume first that $g\in C^1([a,b];\calW)$. 
Then \eqref{eq.kurz-convdiffquot} ensues by the Lebesgue convergence theorem 
and \cite[Proposition 1.10]{KrejciLiero09}. Let now $g\in C([a,b];\calW)$ and 
$\varepsilon>0$ be arbitrary. Then there exists $g_\varepsilon\in 
C^1([a,b];\calW)$ such that $\norm{g-g_\varepsilon}_{C([a,b];\calW)}\leq 
\varepsilon$. Hence, for $h>0$ we obtain using the product rule for finite 
differences
\begin{multline*}
 \int_a^{b-h}\langle f(s),h^{-1}\triangle_h(g-g_\varepsilon)(s)\rangle 
\ds 
\\ = \tfrac1h\left(\int_{b-h}^{b}\langle f,g-g_\varepsilon\rangle\ds
 -\int_a^{a+h}\langle f,g-g_\varepsilon\rangle\ds 
 - \int_a^{b-h}\langle \triangle_h 
f,(g-g_\varepsilon)(s+h)\rangle\ds\right),
\end{multline*}
which implies that 
\begin{multline*}
 \abs{\int_a^{b-h}\langle f(s),h^{-1}\triangle_h(g-g_\varepsilon)(s)\rangle 
\ds}
\\
\leq 
\left(2\norm{f}_{L^\infty(a,b);\calW^*)} 
+ h^{-1}\int_a^{b-h}\norm{\triangle_h f}_{\calW^*}\ds \right)
\norm{g-g_\varepsilon}_{C([a,b];\calW) }\,.
\end{multline*}
Thanks to Lemma \ref{app.lembv2}, the right hand side is bounded by 
$2(\norm{f}_{L^\infty((a,b);\calW^*)} + \Var_{\calW^*}(f,[a,b]))\varepsilon$. 
Standard arguments now finish the proof of \eqref{eq.kurz-convdiffquot} for 
arbitrary $g\in C([a,b];\calW)$.
\end{proof}

\section{Miscellaneous of useful tools}
We collect the statements of results useful for our analysis.
\subsection{Lower semicontinuity properties}

The following Proposition is a slight variant of \cite[Lemma 3.1]{MRS09}. 
\begin{proposition}
 \label{app_prop:lsc} 
Let $v_n,v\in L^\infty(0,S;\calV)$ with $v_n\overset{*}{\rightharpoonup}{v}$ in 
$L^\infty(0,S;\calV)$ and  $\delta_n,\delta\in L^1(0,S;[0,\infty))$ with 
$\liminf_{n\to\infty} \delta_n(s)\geq \delta(s)$ for almost all $s$. 
Then for every $\alpha\geq 1$ 
\begin{align}
 \label{app:eq20}
 \liminf_{n\to\infty}\int_0^S\norm{v_n(s)}_\calV^\alpha \delta_n(s)\ds \geq 
\int_0^S 
\norm{v(s)}_\calV^\alpha\delta(s)\ds.
\end{align}
 \end{proposition}
The next lemma is cited from 
\cite[Lemma 4.3]{MRS12VarConv}.
\begin{lemma}
\label{app_prop:lsc2}
 Let $I\subset\R$ be a bounded interval and $f,g,f_n,g_n:I\to[0,\infty)$, $n\in 
\N$, measurable functions satisfying 
$
 \liminf_{n\to\infty} f_n(s)\geq f(s)$ for a.a.\ $s\in I$ and  
$g_n\rightharpoonup g$  weakly in $L^1(I)$. 
Then
\[
 \liminf_{n\to\infty}\int_I f_n(s)g_n(s)\ds\geq\int_I f(s)g(s)\ds\,.
\] 
\end{lemma}

\subsection{Absolutely continuous functions and $BV$-functions}

\label{app.bvac} 
We follow \cite[Section 2.2]{MRS16}. Let $\calX$ be a Banach space and let 
$\calR:\calX\to\R$ be convex, lower semicontinuous, positively homogeneous of 
degree one and with \eqref{eq.Mief100}. For $1\leq p\leq \infty$, we define the 
set of $p$-absolutely continuous functions (related to $\calR$) as 
\begin{multline}
\label{eq.AC01}
 AC^p([a,b];\calX):=\big\{z:[a,b]\to\calX\,;\,
 \exists m\in L^p((a,b)),\, 
m\geq 0,\, \forall s_1<s_2\in [a,b]: \quad
\big.
\\
\big.
\calR(z(s_2)-z(s_1))\leq\int_{s_1}^{s_2}m(r)\dr\big\}\,.
\end{multline}
Observe that thanks to \eqref{eq.Mief100} this set coincides with the one 
defined with $\norm{\cdot}_\calX$ instead of $\calR$. Let $z\in 
AC^p([a,b];\calX)$. It is shown in \cite[Prop.\ 2.2]{MRS08-metricapproach}, 
\cite[Thm.\ 1.1.2]{AGS05} that for almost every $s\in [a,b]$ the limits
\begin{align*}
 \calR[z'](s):=\lim_{h\searrow0}\calR((z(s+h) - z(s))/h) = 
\lim_{h\searrow0}\calR((z(s) - z(s-h))/h)
\end{align*}
exist and are equal, that $\calR[z']\in L^p((a,b))$ and that $\calR[z']$ is the 
smallest function for which the integral estimate in \eqref{eq.AC01} is valid. 

Let further $\Var_\calR(z;[a,b])$ denote the $\calR$-variation of $z:[a,b]\to 
\calX$, i.e.\ 
\[
\Var_\calR(z;[a,b]):=\sup_{\text{partitions of $[a,b]$}} 
\sum_{i=1}^m \calR(z(s_i)-z(s_{i-1})).
\]
A proof for the next Lemma can be found in  
\cite[Lemma C.1]{KneesThomas2018}. 
\begin{lemma}
 \label{app.lembv1}
For all $p\in (1,\infty]$ and 
$z\in AC^p([a,b];\calX)$ we have 
\begin{align}
 \label{eq.AC02}
 \Var_\calR(z,[a,b])=\int_a^b\calR[z'](s)\ds.
\end{align}
\end{lemma}
The following Lemma is proved in \cite[Theorem 2.20]{Leoni2nded17}:
\begin{lemma}
\label{app.lembv2}
 For every $f\in BV([a,b],\calX)$ we have
 \begin{align}
 \label{eq.lemmbv2}
  \sup_{0<h<(b-a)}h^{-1}\int_a^{b-h}\norm{f(s+h) - f(s)}_\calX\leq 
\Var_\calX(f,[a,b]).
 \end{align}
\end{lemma}

\subsection{A combination of Helly's Theorem and the Ascoli-Arzel{\`a} Theorem}
The general statements of the following theorem can be found in  
\cite{MRS16,AGS05}. For a 
proof tailored to our specific situation we refer to \cite[Proposition 
D.1]{KneesThomas2018}. 
\begin{proposition} 
\label{prop.hellyarzasc-version2}
Let $\calZ$ be a reflexive Banach space, $\calV,\calX$ further Banach spaces 
such that  \eqref{eq.Mief000} is satisfied and assume that 
$\calR:\calX\to [0,\infty)$ complies with  \eqref{eq.Mief100}.   
\begin{itemize}
 \item[(a)] The set $AC^1([a,b];\calX)\cap L^\infty((a,b);\calZ)$ is 
contained in $C([a,b];\calV)$ and there exists $C>0$ such that for all $z\in 
AC^1([a,b];\calX)\cap L^\infty((a,b);\calZ)$ we have 
\[
\norm{z}_{C([a,b];\calV)}\leq  C(\norm{z}_{L^\infty((a,b);\calZ)} 
+\norm{\calR[z']}_{L^1((a,b))}).
\]
\item[(b)] Let $(z_n)_n\subset AC^\infty([a,b];\calX)\cap 
L^\infty((a,b);\calZ)$ be uniformly bounded in the sense that 
$A:=\sup_n\norm{z_n}_{L^\infty((a,b);\calZ)}<\infty$ and 
$B:=\sup_n\norm{\calR[z']}_{L^\infty((a,b))}<\infty$.

Then there exists $z\in AC^\infty([a,b];\calX)\cap 
L^\infty((a,b);\calZ)$ and a (not relabeled) subsequence $(z_n)_n$ such that 
\begin{align}
 z_n&\to z \text{ uniformly in }C([a,b];\calV),
\label{eq:arzelaascolihellymod02}
\\
\forall t\in [a,b]\quad  z_n(t)&\rightharpoonup z(t) \text{ weakly in }\calZ.
\label{eq:arzelaascolihellymod01}
\end{align}
\item[(c)] It is $L^\infty((a,b);\calZ)\cap 
C([a,b];\calV)\subset C_\text{weak}([a,b];\calZ)$.
\end{itemize}
\end{proposition}

\subsection{Chain rule}

The following chain rule is proved in \cite[Prop.~E.1]{KneesThomas2018}.

\begin{proposition}
 \label{prop.chainrules}
Let $z\in H^1((0,T);\calV)\cap L^\infty((0,T);\calZ)$ and assume that  
$\rmD\calI(z(\cdot))\in L^\infty((0,T);\calV^*)$. Then for almost all $t$, 
the mapping $t\mapsto\calI(z(t))$ is differentiable and we have the identity 
\begin{align*}
 \frac{\rmd}{\rmd t}\calI(z(t))=\langle A z(t),\dot 
z(t)\rangle_{\calV^*,\calV} + \langle\rmD\calF(z(t)),\dot 
z(t)\rangle_{\calV^*,\calV}\,.
\end{align*}

Integrated version of the chain rule: 
Let $z\in W^{1,1}((0,T);\calV)\cap 
L^\infty((0,T);\calZ)$ with  $\rmD\calI(z(\cdot))\in L^\infty((0,T);\calV^*)$ 
and assume that $t\mapsto\calI(z(t))$ is continuous on $[0,T]$. Then for all 
 $t_1<t_2\in [0,T]$
\begin{align}
\label{eq.int.chain.rule}
 \calI(z(t_2))-\calI(z(t_1))=\int_{t_1}^{t_2}\langle \rmD\calI(z(r)),\dot 
z(r)\rangle_{\calV^*,\calV}\dr.
\end{align}
\end{proposition}

\end{appendix}

\section*{Acknowledgments} 
The authors are grateful to Alex Mielke for the inspiring discussions and the 
permanent scientific support.
This research has been partially funded by Deutsche 
Forschungsgemeinschaft 
(DFG) through the Priority Programme SPP 1962 Non-smooth and 
Complementarity-based Distributed Parameter Systems: Simulation and 
Hierarchical Optimization, Project P13 Simulation and Optimization of 
Rate-Independent Systems with Non-Convex Energies. 
CZ is a member of GNAMPA-INdAM. DK acknowledges the kind 
hospitality of the DISMA, Politecnico di Torino, and CZ acknowledges the kind 
hospitality of the University of Kassel.

\small

\end{document}